\documentclass[10pt,a4paper,abstract=true,sfdefaults=false]{scrartcl}

\usepackage[T1]{fontenc}
\usepackage[utf8]{inputenc}
\usepackage[english]{babel}
\usepackage[autostyle]{csquotes}
\usepackage{lmodern}

\usepackage{amsmath, amsfonts, amssymb, amsthm, mathtools}

\numberwithin{equation}{section}

\usepackage[a4paper, margin=1in]{geometry}
\usepackage{microtype}

\usepackage{xcolor}
\usepackage[colorlinks=true, linkcolor=blue, citecolor=blue, urlcolor=blue]{hyperref}
\usepackage{xurl}

\usepackage{enumitem}

\setkomafont{disposition}{\normalfont\bfseries}
\setkomafont{title}{\normalfont\Large\bfseries}
\setkomafont{subtitle}{\normalfont\large}
\setkomafont{author}{\normalfont}
\setkomafont{date}{\normalfont}
\setkomafont{publishers}{\normalfont\small}

\usepackage[backend=biber,style=alphabetic,sorting=ynt]{biblatex}
\appto\bibsetup{\sloppy\emergencystretch=1em}
\usepackage[automark,pagestyleset=standard,markcase=upper]{scrlayer-scrpage}
\usepackage{zref-clever} 
\zcsetup{cap}

\theoremstyle{definition}

\newtheorem{theorem}{Theorem}[section]

\newtheorem{definition}[theorem]{Definition}
\newtheorem{lemma}[theorem]{Lemma}
\newtheorem{remark}[theorem]{Remark}

\newtheorem{example}[theorem]{Example}
\newtheorem{proposition}[theorem]{Proposition}
\newtheorem{conjecture}[theorem]{Conjecture}

\makeatletter
\renewenvironment{proof}[1][Proof]{%
  \par\pushQED{\qed}%
  \normalfont
  \trivlist
  \item[\hskip\labelsep\bfseries #1\@addpunct{.}]%
}{%
  \popQED\endtrivlist\ignorespacesafterend
}
\renewcommand*\@pnumwidth{2em} 
\makeatother

\DeclareMathOperator*{\supp}{supp}

\DeclareMathOperator{\Ima}{Im}

\DeclareMathOperator{\divg}{div}

\DeclareMathOperator{\trace}{trace}
\DeclareMathOperator*{\rank}{rank}

\DeclareMathOperator{\Sym}{Sym}
\DeclareMathOperator{\sym}{sym}
\DeclareMathOperator{\spanset}{span}
\DeclareMathOperator{\intset}{int}

\DeclareMathOperator{\rot}{rot}
\DeclareMathOperator{\Lip}{Lip}

\newcommand{\restrict}[1]{\left.\right|_{#1}}

\newcommand{\D}{\mathbb D}
\newcommand{\R}{\mathbb R}
\newcommand{\C}{\mathbb C}

\newcommand{\N}{\mathbb N}

\newcommand{\Sp}{\mathbb S}
\newcommand{\dd}{\mathop{}\!\mathrm{d}}
\newcommand{\CAT}{\mathrm{CAT}}
\renewcommand{\Re}{\operatorname{Re}}
\renewcommand{\Im}{\operatorname{Im}}

\title{Nowhere Locally Metric-Minimising Saddle Surfaces}
\author{Stefan Christian Kohlmeier}
\date{\today}
\publishers{}

\begin{document}

\maketitle

\begin{abstract}
    We study the relation between the saddle property and local metric minimality for smooth surfaces in Euclidean spaces. Anton Petrunin and Stephan Stadler proved that every smooth strictly saddle surface in \(\R^3\) is locally metric-minimising and expressed the expectation that this phenomenon fails in \(\R^4\). We confirm this expectation and prove a stronger density result: Every smooth strictly saddle embedding of a closed disc into \(\R^4\) can be approximated arbitrarily closely in the \(C^\infty\)-topology by smooth strictly saddle embeddings that are nowhere locally metric-minimising. In particular, nowhere locally metric-minimising saddle embeddings exist in \(\R^d\) for every \(d \geq 4\). The main ingredient is the construction of non-trivial infinitesimal isometric deformations of saddle immersions in \(\R^4\) supported in arbitrarily small discs. We also provide an explicit twisted quadratic surface that is saddle in a neighbourhood of the origin but fails to be locally metric-minimising there, by deriving a necessary condition for metric minimality.
\end{abstract}

\section{Introduction}\label{sec:introduction}

For smooth surfaces in \(\R^3\), Anton Petrunin and Stephan Stadler proved that every strictly saddle surface is locally metric-minimising. Their proof reduces the problem to an energy-minimisation problem and leads to a semilinear hyperbolic system of partial differential equations; see \cite{petruninstadler2019metric}. They also considered the analogous problem in \(\R^4\), but were unable to establish the existence of a solution to the resulting system. In \cite[Section 10]{petruninstadler2019metric}, they therefore expressed the expectation that there exist smooth saddle surfaces in \(\R^4\) that are not locally metric-minimising. In this paper, we confirm this expectation and prove an even stronger result, as stated below.

The notion of a metric-minimising surface was introduced by Petrunin in \cite{petrunin1999metric} and subsequently developed by Petrunin and Stadler in \cite{petruninstadler2019metric}. One motivation for studying these surfaces is Samuel Shefel's conjecture.

\begin{conjecture}[Shefel, {\cite{shefel1964intrinsic,petruninstadler2019metric}}]
    Every saddle disc in \(\R^3\) equipped with its intrinsic length metric is a \(\CAT(0)\)-space.
\end{conjecture}

Petrunin and Stadler proved in \cite{petruninstadler2019metric} that every metric-minimising disc in a \(\CAT(0)\)-space is intrinsically \(\CAT(0)\). Together with their result for smooth strictly saddle surfaces in \(\R^3\), they find that every strict saddle disc is locally a \(\CAT(0)\)-surface.

More recently, Alexander Lytchak and Sophia Wagner established the corresponding result for upper curvature bounds in \(\CAT(\kappa)\)-spaces; see \cite{lytchakwagner2024curvature}. The class of metric-minimising discs includes several well-studied classes of maps, such as ruled discs considered by Alexander Alexandrov; see \cite{alexandrov1957ruled,naganoshioyayamaguchi2024two}, harmonic maps from discs into \(\CAT(0)\)-spaces in the sense of Nicholas Korevaar and Richard Schoen; see \cite{korevaarschoen1993sobolev}, and minimal discs in the sense of Alexander Lytchak and Stefan Wenger; see \cite{lytchakwenger2016area}.

The paper is organised as follows. In \zcref{sec:saddle_and_metric-minimising_surfaces}, we recall the notions of saddle and metric-minimising surfaces, and study the quadratic model
\begin{equation}\label{introduction:eq:quadratic_model}
    s: \C  \to \C^2, \ s(z) = (z,z^2)
\end{equation}
and the twisted quadratic model
\begin{equation}\label{introduction:eq:twisted_quadratic_model}
    s: \C \to \C^2, \ s(z) = (z,z^2 e^{i|z|^2}).
\end{equation}

In \zcref{sec:infinitesimal_deformations}, we study infinitesimal deformations of smooth immersions. There are two complementary mechanisms in our approach.

Let \(s: \Omega \to \R^d\) be a smooth immersion, and consider the fibrewise defined bundle map
\begin{equation*}
    \mathcal{B}_p : (N_s\Omega)_p \to \Sym^2(T_p^*\Omega), \ \mathcal{B}_p(\nu) := II^\nu_p.
\end{equation*}
Suppose that \(\rank \mathcal{B}\) = 2 on an open subset \(U \subset \Omega\). The annihilators of the image of \(\mathcal{B}_p\) then form a smooth line bundle \(\mathcal{A}\) on \(U\) with fibres
\begin{equation*}
    \mathcal{A}_p := \{A \in \Sym^2(T_p\Omega) : \langle A,H\rangle_c := \trace(\Phi_A\circ \Psi_H) = 0 \ \text{for every} \ H \in \Ima \mathcal{B}_p\}.
\end{equation*}
The functions \(\Phi_A\) and \(\Psi_H\) are the contractions
\begin{equation*}
    \Phi_A: T_p^*\Omega \to T_p\Omega, \ \Phi_A(\omega) = \iota (A(\omega, \cdot)) = A^{ij}\omega_i\partial_j
\end{equation*}
and
\begin{equation*}
    \Psi_H : T_p\Omega \to T_p^*\Omega, \ \Psi_H(X) = H(X, \cdot) = H_{ij}X^idx_j,
\end{equation*}
where \(\iota: (T_p^*\Omega)^* \to T_p\Omega\) is the canonical identification.

After possibly shrinking \(U\), choose a smooth nowhere vanishing section \(A \in \Gamma(\mathcal{A})\). Suppose that \(A\) is non-degenerate, set
\begin{equation*}
    \alpha_A := -(\Phi_A)^{-1} \divg A,
\end{equation*}
and define the \emph{twist form} \(\Theta \in \Omega^2(U)\) by \(\Theta := d\alpha_A\).

\begin{theorem}\label{introduction:thm:criterion_locally_metric-minimising}
    If the map \(s\) is locally metric-minimising at \(p \in U\), then \(\Theta_p = 0\).
\end{theorem}

This criterion partially answers Question 9.1 from \cite{lytchakwagner2024curvature}. Moreover, it applies to the twisted quadratic model \eqref{introduction:eq:twisted_quadratic_model} in a small disc around the origin. On the other hand, the twisted quadratic model is strictly saddle near the origin. We thereby obtain an explicit example of the phenomenon expected by Petrunin and Stadler.

The second mechanism rests on the following proposition, which is also of independent interest.

\begin{proposition}\label{introduction:prop:existence_infinitesimal_isometric_deformations}
    Let \(s:\Omega \to \R^d\), \(d \geq 4\) be a smooth saddle immersion, and let \(U \Subset \Omega\) be a non-empty open set. Then there exists a non-trivial infinitesimal isometric deformation \(Y \in C_c^\infty(U, \R^d)\) of \(s\). Moreover, its support may be chosen inside an arbitrarily small disc contained in \(U\).
\end{proposition}

If \(Y\) is such a deformation and \(s_t = s + tY\), then the first-order variation of the induced metric vanishes and
\begin{equation*}
    g_{s_t} - g_s = t^2 (DY)^\top DY.
\end{equation*}
Together with suitable rank properties established in \zcref{sec:infinitesimal_deformations}, this allows us to perturb a strictly saddle embedding successively on arbitrarily small discs.

Finally, in \zcref{sec:the_density_theorem}, we iterate this construction over a countable basis of the topology of the domain, and prove the following density theorem.

\begin{theorem}\label{introduction:thm:density_theorem}
    Let \(\Omega \subset \C\) be an open disc and let \(s_0 : \overline{\Omega} \to \C^2\) be a smooth, strictly saddle embedding on \(\overline{\Omega}\). Then \(s_0\) can be approximated arbitrarily closely in the \(C^\infty\)-topology by smooth, strictly saddle embeddings that are nowhere locally metric-minimising in \(\Omega\).
\end{theorem}

\paragraph{Acknowledgements.} I thank Anton Petrunin and Stephan Stadler for proposing this problem and for many helpful discussions and comments on this work. I am particularly grateful to Anton Petrunin for suggesting the twisted quadratic model \eqref{introduction:eq:twisted_quadratic_model} as well as for the idea of introducing a two form measuring the twist.

\section{Saddle and Metric-Minimising Surfaces}\label{sec:saddle_and_metric-minimising_surfaces}

We recall the notions of saddle and metric-minimising surfaces following Petrunin and Stadler \cite{petruninstadler2019metric}. Unless stated otherwise, we denote by \(\langle\cdot,\cdot\rangle\) the Euclidean inner product in \(\R^d\) for some integer \(d \geq 2\). Throughout the paper, \(\mathbb D\) denotes a closed disc and \(D:=\operatorname{int}\mathbb D\) its interior. 

 The proofs in this section might be familiar to experts in the field, so they are omitted. A great book to consult is \cite{zamorapetrunin2020what}. 

\subsection{Saddle Surfaces}\label{subsec:saddle_surfaces}

\begin{definition}\label{def:saddle}
    Let \(\Omega \subset \R^2\) be an open subset. We say that a map \(s: \Omega \to \R^d\) is \emph{saddle at a point} \(p \in \Omega\) if there exists a disc \(\D\) centred at \(p\) such that, for any affine hyperplane \(\Pi \subset \R^d\), each connected component of \(\D \setminus s^{-1}(\Pi)\) meets the boundary \(\partial\D\). We call the map \(s\) \emph{saddle} in \(\Omega\) if it is saddle at every point \(p \in \Omega\).
\end{definition}

\begin{example}\label{ex:not_saddle_at_the_origin}
    The paraboloid \(s: \R^2 \to \R^3, \ s(x,y)=(x,y,x^2+y^2)\) is not saddle at the origin.
\end{example}

Smooth saddle surfaces can be studied through their second fundamental form. Let \(s:\Omega \to \R^d\) be a \(C^1\)-immersion. For \(p \in \Omega\), the \emph{first fundamental form} is given by
\begin{equation*}
    I_p:T_p\Omega \times T_p\Omega \to \R, \ I_p(v,w) := \langle Ds_p(v),Ds_p(w)\rangle.
\end{equation*}
This is exactly the pullback of the Euclidean metric under \(s\), i.e. the induced Riemannian metric on \(\Omega\). Now let us assume that \(s\) is a \(C^2\)-immersion, and set
\begin{equation*}
    T_s \Omega := Ds(T\Omega) \subset s^*(T\R^d) \quad \text{and} \quad N_s \Omega := (T_s\Omega)^\perp \subset s^*(T\R^d) 
\end{equation*}
be the induced tangent- and normal bundles. Thus \(s^*(T\R^d) = T_s \Omega \oplus N_s \Omega\). The \emph{second fundamental form} is given by
\begin{equation*}
    II_p:T_p\Omega \times T_p\Omega \to (N_s \Omega)_p, \ II_p(v,w) := (D^2s_p(v,w))^\perp.
\end{equation*}
For a vector \(\nu \in (N_s \Omega)_p\), we also consider the \emph{scalar second fundamental form}
\begin{equation*}
    II_p^\nu : T_p\Omega \times T_p\Omega \to \R, \ II_p^\nu(v,w) = \langle II_p(v,w),\nu\rangle = \langle D^2s_p(v,w),\nu\rangle.
\end{equation*}

\begin{definition}\label{def:strictly_saddle}
    Let \(\Omega \subset \R^2\) be an open subset. We call a \(C^2\)-immersion \(s:\Omega \to \R^d\) \emph{strictly saddle at a point \(p\)} if the scalar second fundamental form \(II_p^\nu\) is indefinite for all \(\nu \in (N_s \Omega)_p\setminus\{0\}\). We call the map \(s\) \emph{strictly saddle in \(\Omega\)} if it is strictly saddle at every point \(p \in \Omega\).
\end{definition}

The following proposition links the notions of saddle and strictly saddle surfaces introduced in \zcref{def:saddle,def:strictly_saddle}.

\begin{proposition}\label{prop:saddle_second_fundamental_form}
    \(s\) is saddle in \(\Omega\) if and only if \(\det II_p^\nu \leq 0\) for all \(p \in \Omega\) and \(\nu \in (N_s\Omega)_p \setminus \{0\}\). In particular, we have for a fixed point \(p \in \Omega\):
    \begin{enumerate}[label=(\alph*)]
        \item If \(s\) is strictly saddle at \(p\), then \(s\) is saddle at \(p\).
        \item If \(II_p^\nu\) is non-degenerate for every \(\nu\in (N_s \Omega)_p\setminus\{0\}\), then \(s\) is saddle at \(p\) if and only if it is strictly saddle at \(p\).
    \end{enumerate}
\end{proposition}

\subsection{Metric-Minimising Surfaces}

We denote by \(\ell(\gamma)\) the length of a curve \(\gamma\).

\begin{definition}\label{def:metric_minimising}
    Let \(\D \subset \C\) be a disc and \(s: \D \to \R^d\) be a Lipschitz embedding. We say that \(s\) is \emph{metric-minimising on \(\D\)} if its intrinsic metric is minimal. Explicitly, this means that if a continuous map \(s':\D \to \R^d\) agrees with \(s\) on \(\partial\D\) and satisfies
    \begin{equation}\label{eq:def:metric_minimising:1}
        \ell(s\circ\gamma) \geq \ell(s'\circ\gamma) \quad \text{for all Lipschitz curves} \ \gamma:[0,1] \to \D,
    \end{equation}
    then equality holds in \zcref{eq:def:metric_minimising:1}.

    Let \(\Omega \subset \R^2\) be an open set and \(s:\Omega\to\R^d\) be a continuous map. We say that \(s\) is \emph{locally metric-minimising at \(p \in \Omega\)} if there exists a disc \(\D\) centred at \(p\) such that \(s\restrict{\D}\) is a metric-minimising Lipschitz embedding.
\end{definition}

\begin{lemma}\label{lem:metric_minimising_riemannian_metric}
    Let \(s,s':\Omega\to\R^d\) be two continuously differentiable maps, where \(\Omega \subset \R^2\) is an open set. Let us denote by \(g_s = s^*g_{\R^d}\) and \(g_{s'} = (s')^*g_{\R^d}\) the induced pullback metrics, respectively. Then the following statements are equivalent:
    \begin{enumerate}[label=(\alph*)]
        \item We have
        \begin{equation*}
            \ell(s\circ\gamma) \geq \ell(s'\circ\gamma) \quad \text{for all Lipschitz curves} \ \gamma:[0,1] \to \Omega.
        \end{equation*}
        \item \(g_s \geq g_{s'}\) pointwise as pullback quadratic forms.
    \end{enumerate}
\end{lemma}

The following observation is stated in \cite{petruninstadler2019metric} without a proof. We provide a proof in \zcref{appendix:sec:proofs_saddle_and_metric_minimising_surfaces}.

\begin{proposition}\label{prop:locally_metric-minimising_saddle}
    Let \(s:\Omega \to \R^d\) be a locally metric-minimising map at \(p \in \Omega\). Then \(s\) is saddle at \(p\).
\end{proposition}

\subsection{Motivation and Examples}\label{subsec:motivation_and_examples}

We now discuss two examples that illustrate the relation between the saddle property and local metric minimality. 

\subsubsection{The Quadratic Model}\label{subsubsec:the_quadratic_model}

We consider the map 
\begin{equation}\label{eq:subsubsec:the_quadratic_model:1}
    s: \C \to \C^2, \ s(z) = (z,z^2).
\end{equation}
We are going to show that this map is strictly saddle and locally metric-minimising.\smallskip

To prove the strict saddle property, we write \(z = x + iy\) and obtain \(s(x,y) = (x,y,x^2-y^2,2xy)\). A frame for the tangent space is given by
\begin{equation*}
    s_x = (1,0,2x,2y) \quad \text{and} \quad s_y = (0,1,-2y,2x),
\end{equation*}
and a suitable orthogonal frame for the normal space is
\begin{equation*}
    N_1 = (-2x,2y,1,0) \quad \text{and} \quad N_2 = (-2y,-2x,0,1).
\end{equation*}
Hence any normal vector field \(\nu\) admits the form
\begin{equation*}
    \nu = \lambda_1 N_1 + \lambda_2 N_2 \quad \text{for some} \ \lambda_1, \lambda_2 \in C^\infty(\R^2).
\end{equation*}
A computation yields
\begin{equation*}
    II^\nu = 2\begin{pmatrix}
        \lambda_1 & \lambda_2\\
        \lambda_2 & -\lambda_1
    \end{pmatrix}.
\end{equation*}
For \(\nu \neq 0\), its determinant is \(\det(II^\nu) = -4(\lambda_1^2 + \lambda_2^2)\). Thus the scalar second fundamental form is indefinite, so the map \(s\) is strictly saddle at every point.\smallskip

In order to show the metric-minimising property, we use that the complex components of the map from \zcref{eq:subsubsec:the_quadratic_model:1} are holomorphic. Hence all components of \(s\), written in real coordinates, are harmonic. Therefore, for every disc \(\D \subset \C\), the restriction \(s\restrict{\D}\) is the unique minimiser of the Dirichlet energy
\begin{equation*}
    E(h) := \frac{1}{2}\int_{\D} |dh|^2
\end{equation*}
among all maps \(h \in W^{1,2}(\D, \C^2)\) with \(h\restrict{\partial\D} = s\restrict{\partial\D}\). Now let \(f:\D \to \C^2\) be a continuous map that agrees with \(s\) on \(\partial \D\), and suppose that
\begin{equation}\label{eq:subsubsec:the_quadratic_model:2}
    \ell(f\circ\gamma) \leq \ell(s\circ\gamma) \quad \text{for all Lipschitz curves} \ \gamma:[0,1] \to \D.
\end{equation}
Then \(f\) is Lipschitz and \(g_f \leq g_s\) almost everywhere. This implies \(|df|^2 \leq |ds|^2\), so
\begin{equation*}
    E(f) \leq E(s).
\end{equation*}
Since \(s\) is the unique minimiser of the Dirichlet energy, we conclude that \(f = s\). In particular, we have equality in \eqref{eq:subsubsec:the_quadratic_model:2}, thus \(s\) is locally metric-minimising.

\begin{remark}\label{rem:dirichlet_energy}
    This energy argument is also used in \cite[Proposition 1.2]{petruninstadler2019metric} to prove that every smooth strictly saddle surface in \(\R^3\) is locally metric-minimising.
\end{remark}

\subsubsection{A Twisted Local Model}\label{subsubsec:a_twisted_local_model}

Now we consider a twisted version of \eqref{eq:subsubsec:the_quadratic_model:1}: 
\begin{equation}\label{eq:subsubsec:a_twisted_local_model:1}
    s:\C \to \C^2, \ s(z) = (z,z^2e^{i|z|^2}).
\end{equation}
Geometrically, the idea is that the twist keeps the saddle property around the origin but deforms lengths. \smallskip

We first study the saddle property by computing the scalar second fundamental form. By rotational symmetry, it suffices to consider points \(p \in [0,\infty)\). Let us introduce the auxiliary function
\begin{equation*}
    f(z) = (z+p,(z+p)^2e^{i|z+p|^2}).
\end{equation*}
In order to analyse \(s\) at \(p\), it suffices to analyse \(f\) at the origin. We write \(z = x + iy\) and obtain
\begin{equation*}
    f(x,y) = (x+p, y, u(x,y), v(x,y)).
\end{equation*}
Now let us set
\begin{equation*}
    g(x,y) := u(x,y) + iv(x,y) = (x+p+iy)^2e^{ir^2} \quad \text{and} \quad r=\sqrt{(x+p)^2 + y^2}.
\end{equation*}
Let us consider the normal vector fields
\begin{equation*}
    N_1 = (-u_x,-u_y,1,0) \quad \text{and} \quad N_2 = (-v_x, -v_y, 0, 1).
\end{equation*}
Then any normal vector field \(\nu\) admits the form 
\begin{equation*}
    \nu = \lambda_1 N_1 + \lambda_2 N_2 \quad \text{for} \ \lambda_1, \lambda_2 \in C^\infty(\R^2).
\end{equation*}
Moreover, we observe the relation
\begin{equation}\label{eq:subsubsec:saddle_surfaces:3}
    II_p^\nu = \lambda_1 D^2u + \lambda_2 D^2v = \lambda_1 \Re(D^2g(0,0)) + \lambda_2 \Im(D^2g(0,0)).
\end{equation}
A lengthy but straightforward computation yields
\begin{equation}
    D^2g(0,0) = e^{ip^2}\begin{pmatrix}
        2-4p^4+10ip^2 & -4p^2+2i\\
        -4p^2+2i & -2+2ip^2
    \end{pmatrix} = 2e^{ip^2} \left(\begin{pmatrix}
        1 & i\\
        i & -1
    \end{pmatrix} + O(p^2)\right).
\end{equation}
By continuity, every non-zero scalar second fundamental form is indefinite in a sufficiently small disc \(\D\) around the origin. Hence \(s\) is strictly saddle, and therefore saddle by \zcref{prop:saddle_second_fundamental_form}, in \(\D\).

Now choose \(\lambda_1 = - \sin(p^2)\) and \(\lambda_2 = \cos(p^2)\) in \zcref{eq:subsubsec:saddle_surfaces:3}. This choice removes the phase factor and extracts the imaginary part, yielding
\begin{equation*}
    II_p^\nu = \Im \begin{pmatrix}
    2-4p^4+10ip^2 & -4p^2+2i\\
    -4p^2+2i & -2+2ip^2
    \end{pmatrix}
    = \begin{pmatrix}
        10p^2 & 2\\
        2 & 2p^2
    \end{pmatrix}.
\end{equation*}
For sufficiently large \(p\), this second fundamental form is positive definite. Hence, by \zcref{prop:saddle_second_fundamental_form}, the map \(s\) is not saddle sufficiently far from the origin.\smallskip

Now we show that the map \zcref{eq:subsubsec:a_twisted_local_model:1} is not locally metric-minimising at the origin. We consider a disc \(\D\) of arbitrary radius \(R > 0\) with centre \(0\) and compute the Riemannian metric induced by \(s\) on \(\D\). To do so, we write in polar coordinates \(s(r,\theta) = (re^{i\theta}, r^2e^{i(2\theta + r^2)})\). We compute
\begin{align*}
    s_r &= (e^{i\theta}, (2r+2ir^3)e^{i(2\theta + r^2)})\\
    s_\theta &= (ire^{i\theta}, 2ir^2e^{i(2\theta + r^2)})\\
    |s_r|^2 &= 1 + 4r^2 + 4r^6\\
    |s_\theta|^2 &= r^2+4r^4\\
    \langle s_r,s_\theta\rangle &= \Re\left(e^{i\theta} \cdot \overline{ire^{i\theta}} + (2r + 2ir^3)e^{i(2\theta + r^2)} \cdot \overline{2ir^2e^{i(2\theta + r^2)}}\right) = 4r^5.
\end{align*}
We make the change of coordinates \(\theta = \vartheta + \psi(r)\) in order to obtain a diagonal Riemannian metric. We set \(\tilde{s}(r, \vartheta) := s(r,\vartheta + \psi(r))\), and compute
\begin{align*}
    \tilde{s}_r &= s_r + \psi'(r)s_\theta\\
    \tilde{s}_\vartheta &= s_\theta\\
    \langle \tilde{s}_r, \tilde{s}_\vartheta\rangle &= \langle s_r, s_\theta\rangle + \psi'(r)|s_\theta|^2.
\end{align*}
Now we choose \(\psi'(r) = -\frac{4r^5}{r^2+4r^4}\), and the Riemannian metric becomes diagonal with
\begin{equation}\label{eq:prop:not_locally_metric-minimising_origin:1}
    g_{\tilde{s}} = \left(1+4r^2+\frac{4r^6}{1+4r^2}\right)dr^2 + (r^2+4r^4)d\vartheta^2.
\end{equation}
We want \(\psi(R) = 0\) in order to preserve the boundary values, i.e.
\begin{equation*}
    \psi(r) = \int_{r}^R \frac{4t^5}{t^2+4t^4} \dd t.
\end{equation*}
Since \(\psi\) is a smooth function of \(r^2\) near \(r=0\), this change of coordinates extends smoothly across the origin. Now let us consider the map \(\tilde{s}'(r,\vartheta) = (re^{i\vartheta}, r^2e^{i(2\vartheta + R^2)})\). The preceding calculations give
\begin{equation*}
    g_{\tilde{s}'} = (1+4r^2)dr^2 + (r^2+4r^4)d\vartheta^2.
\end{equation*}
We clearly have \(\tilde{s}\restrict{\partial\D} = \tilde{s}'\restrict{\partial\D}\) and
\begin{equation*}
    g_{\tilde{s}} - g_{\tilde{s}'} = \frac{4r^6}{1+4r^2} dr^2 \geq 0.
\end{equation*}
The radial curve \(\gamma:[0,R] \to \D, \gamma(r) = r\) gives the strict inequality \(\ell(\tilde{s} \circ \gamma) > \ell(\tilde{s}' \circ \gamma)\), so \(\tilde{s}\) is not locally metric-minimising at the origin. Let 
\begin{equation*}
    s'(r,\theta) = \tilde{s}'(r, \theta - \psi(r))
\end{equation*}
be the map resulting from reversing the change of coordinates above. Therefore, \(s\restrict{\partial\D} = s'\restrict{\partial\D}\). According to \zcref{lem:metric_minimising_riemannian_metric},
\begin{equation*}
    \ell(s \circ \gamma) \geq \ell(s' \circ \gamma) \quad \text{for all Lipschitz curves} \ \gamma:[0,1] \to \D.
\end{equation*}
The inequality is strict along a curve. Thus \(s\) is not locally metric-minimising at the origin. In \zcref{subsubsec:an_application_to_the_twisted_local_model} we will even prove that \(s\) is not locally metric-minimising in a small disc around the origin.

\section{Infinitesimal Deformations}\label{sec:infinitesimal_deformations}

In this section, we study infinitesimal deformations, which we use both to provide a counterexample to local metric minimality and to prove \zcref{introduction:thm:density_theorem}.

Let \(\Omega \subset \R^2\) be an open subset and \(s:\Omega \to \R^d\) be a \(C^1\)-immersion. We denote by \(g_s := s^*g_{\R^d}\) the induced Riemannian metric. If it is clear what is meant, we drop the lower index \(s\). For a tangent vector \(X\) and a \(1\)-form \(\omega\) we write \(X^\flat := g(X,\cdot)\in T^*\Omega\) (lowering the index) and \(\omega^\sharp := (g^\flat)^{-1}(\omega)\in T\Omega\) (raising the index).

\begin{definition}\label{def:infinitesimal_metric_operator}
    Let \(Y : \Omega \to \R^d\) be a vector field along \(s\), and consider the variation \(s_t := s + tY\). We define the \emph{infinitesimal metric operator}
    \begin{equation}\label{eq:def:infinitesimal_metric_operator:1}
        \mathcal{E} : \Gamma(s^*(T\R^d)) \to \Gamma(\Sym^2(T^*\Omega)), \ \mathcal{E}(Y) := \frac{1}{2} \frac{d}{dt}\restrict{t = 0} g_{s_t}.
    \end{equation}
\end{definition}

Explicitly, \zcref{eq:def:infinitesimal_metric_operator:1} means that
\begin{equation*}
    \mathcal{E}(Y)(X,Z) = \frac{1}{2} \left(\langle Ds(X), DY(Z)\rangle + \langle Ds(Z), DY(X)\rangle\right) \quad \text{for all} \ X,Z \in \Gamma(T\Omega).
\end{equation*}
Assume from now on that \(s\) is a \(C^2\)-immersion. We define the bundle map \(\mathcal{B} : N_s\Omega \to \Sym^2(T^*\Omega)\) fibrewise by \(\mathcal{B}_p(\nu) = II_p^\nu\). Explicitly, this means for fixed \(p \in \Omega\)
\begin{equation*}
    \mathcal{B}_p(\nu)(X,Z) = \langle II_p (X,Z), \nu\rangle \quad \text{for all} \ \nu \in (N_s\Omega)_p, \ X,Z \in T_p \Omega.
\end{equation*}
Given a vector field \(Y \in \Gamma(s^*(T\R^d))\), let \(Y = Ds(V) + \nu\) be the decomposition of \(Y\) in its tangential and normal parts. By the Gauss formula \cite[Theorem 8.2]{lee2018introduction}
\begin{equation*}
    D(Ds(V))(Z) = Ds(\nabla_Z V) + II(Z,V),
\end{equation*}
while
\begin{equation*}
    \mathcal{E}(Y)(X,Z) = \frac{1}{2}(g(X, \nabla_Z V) + g(\nabla_XV, Z)) - II^\nu(X,Z),
\end{equation*}
since
\begin{equation*}
    0 = Z\langle Ds(X), \nu\rangle = \langle D(Ds(X))(Z), \nu\rangle + \langle Ds(X), D\nu(Z)\rangle = \mathcal{B}(\nu) + \langle Ds(X), D\nu(Z)\rangle.
\end{equation*}
We conclude
\begin{equation}\label{eq:sec:infinitesimal_deformations:1}
    \mathcal{E}(Y) = \sym(\nabla V^\flat)-\mathcal{B}(\nu).
\end{equation}
This formula will be used throughout the section and is adapted from \cite{tenenblat1979infinitesimal}.

\subsection{Inhomogeneous Infinitesimal Deformations}\label{subsec:inhomogeneous_infinitesimal_deformations}

In this subsection, we study the inhomogeneous equation
\begin{equation*}
    \mathcal{E}(Y) = H \quad \text{for} \ H \in \Gamma(\Sym^2(T^*\Omega))
\end{equation*}
and derive a local criterion for the failure of metric minimality.

At first, we note that if the infinitesimal metric operator admits a linear differential right inverse, then \(s\) is nowhere locally metric minimising, which is content of the next proposition.

\begin{proposition}\label{prop:metric_operator_right_inverse_locally_metric-minimising}
    Let \(s:\Omega \to \R^d\) be a smooth immersion and let \(U \subset \Omega\) be open. Suppose that the infinitesimal metric operator admits a linear differential right inverse of finite order
    \begin{equation*}
        R : \Gamma_c(\Sym^2(T^*U)) \to \Gamma_c(s^*(T\R^d)\restrict{U}),
    \end{equation*}
    that is,
    \begin{equation*}
        \mathcal{E}(R(H)) = H \quad \text{for every} \ H \in \Gamma_c(\Sym^2(T^*U)).
    \end{equation*}
    Then \(s\) is nowhere locally metric-minimising in \(U\).
\end{proposition}
\begin{proof}
    Fix \(p \in U\), and let \(\D \Subset U\) be an arbitrary disc centred at \(p\). We choose a non-negative function \(\rho \in C_c^\infty(D)\), positive on a non-empty disc \(\D' \Subset D\), and constants \(C_j\) such that
    \begin{equation*}
        |D^j\rho|^2 \leq C_j \rho \quad \text{for every} \ j \geq 0.
    \end{equation*}
    Set \(H := -\rho g_s\) and \(Y := R(H)\). Since \(R\) is a differential operator, \(\supp Y \subset \supp H \subset D\). Moreover,
    \begin{equation*}
        \mathcal{E}(Y) = -\rho g_s.
    \end{equation*}
    Let \(k\) be the order of \(R\). Since \(H = -\rho g_s\), the first derivatives of \(Y = R(H)\) involve derivatives of \(\rho\) of order at most \(k+1\). The coefficients of \(R\) as well as the derivatives of \(g_s\) are bounded on the disc \(\D\). Hence we find \(C > 0\) such that
    \begin{equation*}
        |DY(X)|^2 \leq C\rho g_s(X,X) \quad \text{for every} \ X \in T\D.
    \end{equation*}
    Now consider \(s_t := s + tY\). Then we obtain for \(X \in T\D\)
    \begin{equation*}
        \begin{split}
            (g_{s_t} - g_s)(X,X) &= 2t \mathcal{E}(Y)(X,X) + t^2 |DY(X)|^2 = -2t\rho g_s(X,X) + t^2|DY(X)|^2\\
            &\leq (-2t + Ct^2)\rho g_s(X,X).
        \end{split}
    \end{equation*}
    Now choose \(t_0 > 0\) such that \(-2t + Ct^2 < 0\) for \(0 < t < t_0\), then
    \begin{equation*}
        g_{s_t} \leq g_s \quad \text{on} \ \D,
    \end{equation*}
    while on \(\D'\), where \(\rho > 0\),
    \begin{equation*}
        g_{s_t}(X,X) < g_s(X,X) \quad \text{for every} \ X \in T\D'\setminus\{0\}.
    \end{equation*}
    Since \(Y\) is compactly supported in \(D\), the map \(s_t = s + tY\) agrees with \(s\) on \(\partial \D\). Moreover, the metric inequality is strict along every non-constant smooth curve contained in \(\D'\). Hence \(s\restrict{\D}\) is not metric-minimising. Since \(\D\) was arbitrary, \(s\) is not locally metric-minimising at \(p\).
\end{proof}

Let \(s: \Omega \to \R^d\) be a smooth immersion and suppose that \(\rank \mathcal{B} = 2\) on an open set \(U \subset \Omega\). Since \(\Sym^2(T_p^*\Omega)\) has dimension \(3\), the image \(\Ima \mathcal{B}_p \subset \Sym^2(T_p^*\Omega)\) has a one-dimensional annihilator. We therefore introduce the line bundle \(\mathcal{A} := \ker (\mathcal{B}^*\restrict{U}) \subset \Sym^2(TU)\), where
\begin{equation*}
    \mathcal{B}^*:\Sym^2(T\Omega) \to (N_s\Omega)^*, \ (\mathcal{B}^*(A))(\nu) := \langle A, \mathcal{B}(\nu)\rangle_c
\end{equation*}
is the dual bundle map. The line bundle \(\mathcal{A}\) is fibrewise given by
\begin{equation*}
    \mathcal{A}_p := \{A \in \Sym^2(T_p\Omega) : \langle A, H\rangle_c = 0 \ \text{for every} \ H \in \Ima\mathcal{B}_p\}.
\end{equation*}
We use here the canonical contraction
\begin{equation*}
    \langle A, H\rangle_c := \trace(\Phi_A \circ \Psi_H) = A^{ij}H_{ij} \quad \text{for} \ A \in \Sym^2(T_p\Omega) \ \text{and} \ H \in \Sym^2(T_p^* \Omega).
\end{equation*}
The functions \(\Phi_A\) and \(\Psi_H\) are the contractions
\begin{equation*}
    \Phi_A: T_p^*\Omega \to T_p\Omega, \ \Phi_A(\omega) = \iota (A(\omega, \cdot)) = A^{ij}\omega_i\partial_j
\end{equation*}
and
\begin{equation*}
    \Psi_H : T_p\Omega \to T_p^*\Omega, \ \Psi_H(X) = H(X, \cdot) = H_{ij}X^idx_j,
\end{equation*}
where \(\iota: (T_p^*\Omega)^* \to T_p\Omega\) is the canonical identification.

After shrinking \(U\), we choose a nowhere-vanishing section \(A \in \Gamma(\mathcal{A})\). Then, by definition
\begin{equation*}
    \langle A_q, \mathcal{B}_q(\nu)\rangle_c = 0 \quad \text{for all} \ \nu \in (N_s\Omega)_q.
\end{equation*}
So, for every \(p \in U\), it follows that \(\mathcal{A}\) gives \(\Ima \mathcal{B}_p \subset \ker (\langle A_p, \cdot \rangle_c) \subset \Sym^2(T_p^*U)\). Both spaces have dimension \(2\): The first because \(\rank \mathcal{B} = 2\) on \(U\), and the second because \(A_p\) defines a non-zero linear functional on the three-dimensional space \(\Sym^2(T_p^*U)\). Hence \(\Ima \mathcal{B}_p = \ker (\langle A_p, \cdot \rangle_c)\).

Now let us write \(Y = Ds(V) + \nu\). In order to find \(\mathcal{B}(\nu) = \sym(\nabla V^\flat) - H\) (compare to \zcref{eq:sec:infinitesimal_deformations:1}), we therefore may study the equation
\begin{equation*}
    \langle A, \sym(\nabla V^\flat) - H\rangle_c = 0.
\end{equation*}
This motivates the study of the first-order differential operator
\begin{equation*}
    L_A : \Omega_c^1(U) \to C_c^\infty(U), \ L_A\omega := \langle A, \sym(\nabla \omega)\rangle_c.
\end{equation*}
We will derive below a criterion to obtain a differential right inverse of this operator; see \zcref{lem:contraction_operator_surjective}, and this will apply by following the preceding argumentation that \(\mathcal{E}\) admits a differential right-inverse operator; see \zcref{lem:right_inverse_operator}.

Suppose that \(A\) is non-degenerate, i. e. \(\Phi_A\) is a bundle isomorphism. We now define the \(1\)-form
\begin{equation*}
    \alpha_A := -(\Phi_A)^{-1}\divg (A), \quad \text{where} \quad \divg A = (\nabla_{i}{A^{ij}})\partial_j \in \Gamma(TU).
\end{equation*}
\begin{definition}\label{def:twist_form}
    We define the \emph{twist form} of \(s\) by \(\Theta := d\alpha_A \in \Omega^2(U)\).
\end{definition}

\begin{remark}\label{rem:twist_form_is_intrinsic}
    Although \(\alpha_A\) depends on the choice of \(A\), the twist form \(\Theta\) does not. Indeed, if \(\tilde{A} = \lambda A\) for a nowhere vanishing function \(\lambda\), then 
    \begin{equation*}
        \divg(\tilde{A}) = \nabla_i(\lambda A^{ij})\partial_j = (\partial_i \lambda)A^{ij}\partial_j + \lambda \nabla_i A^{ij}\partial_j = \Phi_A(d\lambda) + \lambda \divg(A).
    \end{equation*}
    We therefore obtain
    \begin{equation*}
        \alpha_{\tilde{A}} = -(\Phi_{\tilde{A}})^{-1}\divg(\tilde{A}) = - \frac{1}{\lambda}(\Phi_A)^{-1}(\lambda \divg(A) + \Phi_A(d\lambda)) = \alpha_A - \frac{d\lambda}{\lambda} = \alpha_A - d\log|\lambda|,
    \end{equation*}
    and conclude \(d\alpha_{\tilde{A}} = d\alpha_A\).
\end{remark}

In the following two lemmas we will construct a right inverse of the infinitesimal metric operator in order to apply \zcref{prop:metric_operator_right_inverse_locally_metric-minimising}.

\begin{lemma}\label{lem:contraction_operator_surjective}
    Let \(U \subset \Omega\) be open and suppose that \(A \in \Gamma(\Sym^2(TU))\) is smooth, symmetric and non-degenerate. If \(\Theta\) is nowhere vanishing on \(U\), then
    \begin{equation*}
        L_A : \Omega_c^1(U) \to C_c^\infty(U)
    \end{equation*}
    is surjective and admits a first-order differential right inverse.
\end{lemma}
\begin{proof}
    Let \(f \in C_c^\infty(U)\), and for a \(1\)-form \(\omega\), set \(X := \Phi_A(\omega)\). Then \(X^j = A^{ij}\omega_i\). Using the product rule and symmetry of \(A\), we find
    \begin{equation*}
        L_A \omega = \frac{1}{2}A^{ij}\nabla_i\omega_j + \frac{1}{2}A^{ij}\nabla_j \omega_i = A^{ij}\nabla_j \omega_i
    \end{equation*}
    and
    \begin{equation*}
        \divg(\Phi_A(\omega)) = \nabla_j (A^{ij} \omega_i) = (\nabla_j A^{ij})\omega_i + A^{ij}\nabla_j \omega_i.
    \end{equation*}
    So, using the symmetry of \(A\) again, it follows that
    \begin{equation*}
        \divg(X) = \omega(\divg A) + L_A \omega = -\omega(\Phi_A(\alpha_A)) + L_A\omega = - \alpha_A(X) + L_A \omega,
    \end{equation*}
    that is, 
    \begin{equation}\label{eq:lem:contraction_operator_surjective:1}
        L_A \omega = \divg(X) + \alpha_A(X).
    \end{equation}
    Now let \(*\) denote the Hodge star operator associated with \(g\) and the standard orientation of \(\Omega\). Since \(d\alpha_A\) is nowhere vanishing, we find a smooth nowhere vanishing function \(\lambda\) such that
    \begin{equation*}
        d\alpha_A = \lambda dVol_g.
    \end{equation*}
    Since \(\alpha_A(X) = *(\alpha_A \wedge *X^\flat)\) and \(\divg X = *d*X^\flat\); see \cite[Problem 16-21]{lee2012introduction}, it suffices to solve the equation
    \begin{equation*}
        f = *d*X^\flat +  *(\alpha_A \wedge *X^\flat), \quad \text{or, equivalently,} \quad \frac{f}{\lambda}d\alpha_A = d *X^\flat + \alpha_A \wedge*X^\flat.
    \end{equation*}
    The choice \(*X^\flat = d\left(\frac{f}{\lambda}\right) + \frac{f}{\lambda}\alpha_A\) settles the claim. The construction of \(\omega\) depends differentially on \(f\) with order at most one. Moreover, none of the operations involved enlarges supports. Hence \(L_A\) admits a first-order differential right inverse.
\end{proof}

\begin{lemma}\label{lem:right_inverse_operator}
    Let \(s:\Omega \to \R^d\), \(d \geq 4\), be a smooth immersion, and suppose that \(\rank \mathcal{B} = 2\) on an open set \(U \subset \Omega\). Assume that the bundle \(\mathcal{A}\) is generated on \(U\) by a smooth non-degenerate section \(A\), and that the twist form \(\Theta\) is nowhere vanishing on \(U\). Then the infinitesimal metric operator
    \begin{equation*}
        \mathcal{E} : \Gamma_c(s^*(T\R^d)\restrict{U}) \to \Gamma_c(\Sym^2(T^*U))
    \end{equation*}
    admits a linear differential right inverse of order at most \(2\).
\end{lemma}
\begin{proof}
    By \zcref{lem:contraction_operator_surjective}, fix a first-order differential right inverse
    \begin{equation*}
        Q_A : C_c^\infty(U) \to \Omega_c^1(U) \quad \text{with} \ L_A(Q_A(f)) = f \quad \text{for every} \ f \in C_c^\infty(U).
    \end{equation*}
    Let \(H \in \Gamma_c(\Sym^2(T^*U))\), define
    \begin{equation*}
        f := \langle A,H\rangle_c, \quad \omega := Q_A(f) \quad \text{and} \quad S := \sym(\nabla \omega) - H.
    \end{equation*}
    Since \(Q_A\) is a right-inverse of \(L_A\), we obtain
    \begin{equation*}
        \langle A, S\rangle_c = \langle A, \sym(\nabla \omega)\rangle_c - \langle A, H\rangle_c = L_A \omega - f = 0.
    \end{equation*}
    It follows from the discussion before \zcref{def:twist_form} that \(S \in \Gamma_c(\Ima \mathcal{B}\restrict{U})\). Since \(\mathcal{B}\restrict{U}\) has constant rank, \(\ker \mathcal{B}\restrict U\) is a smooth subbundle of \(N_s\Omega \restrict{U}\); see \cite[Theorem 10.34]{lee2012introduction}. The restriction
    \begin{equation*}
        \mathcal{B}\restrict{\ker(\mathcal{B}\restrict{U})^\perp}:\ker(\mathcal{B}\restrict{U})^\perp \to \Ima \mathcal{B}\restrict{U}
    \end{equation*}
    is therefore a vector bundle isomorphism. Define \( \nu := (\mathcal{B}\restrict{\ker(\mathcal{B}\restrict{U})^\perp})^{-1}(S)\), \(V = \omega^\sharp\),
    and put
    \begin{equation*}
        R(H) := Ds(V) + \nu.
    \end{equation*}
    By construction, we have \(\mathcal{B}(\nu) = S = \sym(\nabla \omega) - H\),
    therefore
    \begin{equation*}
        \mathcal{E}(R(H)) = \sym(\nabla V^\flat) - \mathcal{B}(\nu) = \sym(\nabla \omega) - (\sym(\nabla \omega) - H) = H.
    \end{equation*}
    Thus \(R\) is a right inverse of \(\mathcal{E}\).

    It remains to determine its differential order. The contraction \(H \mapsto \langle A, H\rangle_c\) has order \(0\), the operator \(Q_A\) has order \(1\), and applying \(\sym \nabla\) contributes one additional derivative. All remaining operations have order \(0\), hence \(R\) is a linear differential operator of order at most \(2\). Finally, none of the preceding operations in the construction enlarge supports, therefore \(\supp R(H) \subset \supp H\), and \(R(H)\) is compactly supported in \(U\).
\end{proof}

\begin{theorem}\label{thm:twisting_form_locally_metric_minimising}
    Let \(\Omega \subset \R^2\) be an open subset and let \(s : \Omega \to \R^d\), where \(d \geq 4\), be a smooth immersion. Suppose that \(\rank \mathcal{B} = 2\) on an open subset \(U \subset \Omega\), and that \(\mathcal{A}\) is generated on \(U\) by a smooth non-degenerate section \(A\). If \(s\) is locally metric-minimising at \(p \in U\), then \(\Theta_p = 0\).
\end{theorem}
\begin{proof}
    Suppose, to the contrary, that \(\Theta_p \neq 0\). By continuity, there exists an open neighbourhood \(V \subset U\) of \(p\) on which \(\Theta\) is nowhere vanishing. Applying \zcref{lem:right_inverse_operator} on \(V\), we find that
    \begin{equation*}
        \mathcal{E} : \Gamma_c(s^*(T\R^d)\restrict{V}) \to \Gamma_c(\Sym^2(T^*V))
    \end{equation*}
    admits a differential right inverse of order at most \(2\). By \zcref{prop:metric_operator_right_inverse_locally_metric-minimising}, \(s\) is therefore not locally metric-minimising at \(p\), contradicting the hypothesis. Hence \(\Theta_p = 0\).
\end{proof}

\begin{lemma}\label{lem:twisting_form_graph_coordinates}
    Let \(U \subset \R^2\) be an open set and \(s:U\to\R^d\), \(d \geq 4\), be a smooth immersion. Suppose that \(s\) is given by
    \begin{equation*}
        s(x,y) = (x,y,u^1(x,y),\dots,u^{d-2}(x,y)).
    \end{equation*}
    Assume that \(\rank \mathcal{B} = 2\) and assume that \(\mathcal{A}\) is generated on \(U\) by a smooth non-degenerate section \(A\). Write \(A = A^{ij} \partial_i \otimes \partial_j\), and define the (coordinate) divergence \(\divg_0 A := (\partial_i A^{ij}) \partial_j\). For \(\alpha := -(\Phi_A)^{-1} \divg_0 A\), the twist form \(\Theta\) is given by
    \begin{equation*}
        \Theta = (\partial_x \alpha_2 - \partial_y \alpha_1) dx \wedge dy = \rot \alpha \ dx \wedge dy.
    \end{equation*}
\end{lemma}
\begin{proof}
    See \zcref{appendix:sec:auxiliary_proofs_for_infinitesimal_deformations}.
\end{proof}

\subsubsection{An Application to the Twisted Local Model}\label{subsubsec:an_application_to_the_twisted_local_model}

Here we turn back to the twisted model
\begin{equation}\label{eq:subsubsec:an_application_to_the_twisted_local_model:1}
    s:\C \to \C^2, \ z \mapsto (z,z^2e^{i|z|^2}).
\end{equation}
Write
\begin{equation*}
    s(x,y) = (x,y,u(x,y),v(x,y)), \quad r^2 = x^2 + y^2,
\end{equation*}
and let \(A\) be a smooth generator of the one-dimensional subspace \(\spanset \{D^2u, D^2v\}^\perp \subset \Sym_2(\R)\), normalised so that \(\trace A = 2\). A computation yields
\begin{equation*}
    A(p,0) = \frac{1}{1+3p^4} \begin{pmatrix}
        1 - 2p^4 & p^6 - 3p^2\\
        p^6 - 3p^2 & 1 + 8p^4
    \end{pmatrix} = \begin{pmatrix}
        1 & -3p^2\\
        -3p^2 & 1
    \end{pmatrix} + O(p^4).
\end{equation*}
Using rotational symmetry and conjugating by the corresponding rotation in the domain, we obtain
\begin{equation*}
    A(x,y) = \begin{pmatrix}
        1 + 6xy & 3y^2-3x^2\\
        3y^2-3x^2 & 1-6xy
    \end{pmatrix} + O(r^4).
\end{equation*}
Moreover, we compute
\begin{equation*}
    \divg_0 A = 12 y \partial_x - 12x \partial_y + O(r^3) \quad \text{and} \quad \alpha = -12y dx + 12x dy + O(r^3).
\end{equation*}
It follows that
\begin{equation*}
    \Theta = (24 + O(r^2)) dx \wedge dy.
\end{equation*}
At the origin, \(D^2 u\) and \(D^2 v\) are linearly independent and \(A(0) = I_2\). Hence, after restricting to a sufficiently small disc, \(D^2 u\) and \(D^2 v\) remain pointwise linearly independent and \(\det A \neq 0\). Since \(\Theta = (24 + O(r^2)) dx \wedge dy\), \zcref{thm:twisting_form_locally_metric_minimising} applies, and \eqref{eq:subsubsec:an_application_to_the_twisted_local_model:1} is not locally metric-minimising in a small disc around the origin.

Of course, embedding this surface isometrically in \(\R^d\) for \(d \geq 4\) yields a saddle surface that is not locally metric-minimising either; see \zcref{appendix:lem:existence_not_locally_metric-minimising_saddle_surface_higher_dimension}.

\subsection{Infinitesimal Isometric Deformations}\label{subsec:infinitesimal_isometric_deformations}

In this subsection, we prove the existence of compactly supported infinitesimal isometric deformations of saddle immersions and establish the rank properties needed for the density theorem \zcref{introduction:thm:density_theorem}.

\begin{definition}\label{def:infinitesimal_isometric_deformation}
    Let \(\Omega \subset \R^2\) be an open set, let \(s:\Omega\to\R^d\) be a \(C^1\)-immersion, and let \(Y:\Omega\to\R^d\) be a vector field along \(s\). We call \(Y\) an \emph{infinitesimal isometric deformation} of \(s\) if
    \begin{equation*}
        \mathcal{E}(Y) = 0.
    \end{equation*}
\end{definition}

The main idea behind the proof of \zcref{introduction:thm:density_theorem} is the following observation. If \(Y\) is an infinitesimal isometric deformation of \(s\) and \(s_t:=s+tY\), then
\begin{equation*}
    (g_{s_t}-g_s)(X,X) = 2t\langle Ds(X),DY(X)\rangle+t^2|DY(X)|^2 = t^2|DY(X)|^2.
\end{equation*}
Thus, the induced metric changes only to second order in \(t\).

\begin{lemma}\label{lem:compactly_supported_kernel}
    Let \(D \subset \Omega\) be an open disc and let \(A \in \Gamma(\Sym^2(TD))\) be smooth, nowhere vanishing, and nowhere indefinite. Then there exists a non-zero \(1\)-form \(\omega \in \Omega_c^1(D)\) such that
    \begin{equation*}
        L_A \omega = \langle A,\sym(\nabla\omega)\rangle_c = 0.
    \end{equation*}
    Moreover, the support of \(\omega\) may be chosen in an arbitrarily small disc contained in \(D\).
\end{lemma}
\begin{proof}
    See \zcref{appendix:sec:auxiliary_proofs_for_infinitesimal_deformations}.
\end{proof}

\begin{proposition}\label{prop:existence_infinitesimal_deformation}
    Let \(s:\Omega \to \R^d\), \(d \geq 4\), be a smooth saddle immersion and let \(U \subset \Omega\) be a non-empty open set. Then \(s\) admits a non-trivial compactly supported infinitesimal isometric deformation \(Y \in \Gamma_c(s^*(T\R^d)\restrict{U}\). Its support may be chosen inside an arbitrarily small disc contained in \(U\).
\end{proposition}
\begin{proof}
    Choose an arbitrarily small open disc \(B \Subset U\). Since \(s\) is saddle, we have
    \begin{equation*}
        \rank \mathcal{B}\restrict{B} \leq 2.
    \end{equation*}
    Let
    \begin{equation*}
        r := \max_{q \in B} \rank \mathcal{B}_q \in \{0,1,2\}
    \end{equation*}
    and choose \(p \in B\) with \(\rank \mathcal{B}_p = r\). If \(r = 0\), then \(\mathcal{B}\) has constant rank zero on \(B\). If \(r > 0\), choose a non-zero \(r \times r\)-minor of  \(\mathcal{B}_p\). By continuity, this minor remains non-zero on a neighbourhood \(W \subset B\) of \(p\), hence \(\rank \mathcal{B}\restrict{W} \geq r\). By maximality of \(r\) on \(B\), the reverse inequality also holds, and therefore \(\rank \mathcal{B} \restrict{W} = r\). After choosing a closed disc \(\D \subset W\), we may therefore assume that
    \begin{equation*}
        \rank \mathcal{B}\restrict{\D} = r \in \{0,1,2\}.
    \end{equation*}
    
    \emph{Case 1}: \(r \leq 1\). Since the normal bundle has rank \(d-2\), we have
    \begin{equation*}
        \rank \ker (\mathcal{B}\restrict{\D}) = d - 2 - r \geq 1.
    \end{equation*}
    After shrinking \(\D\) if necessary, choose a smooth nowhere vanishing normal section \(\nu \in \Gamma(\ker \mathcal{B} \restrict{\D})\), pick \(0 \neq \varphi \in C_c^\infty(D)\), and set
    \begin{equation*}
        Y := \varphi \nu.
    \end{equation*}
    Since \(Y\) is normal, we find
    \begin{equation*}
        \mathcal{E}(Y) = -\mathcal{B}(\varphi \nu) = -\varphi \mathcal{B}(\nu) = 0.
    \end{equation*}
    Thus \(Y\) is a non-trivial compactly supported infinitesimal isometric deformation.

    \emph{Case 2:} \(r = 2\). Then
    \begin{equation*}
        \Ima(\mathcal{B} \restrict{\D}) \subset \Sym^2(T^*\Omega\restrict{\D})
    \end{equation*}
    is a rank-two subbundle. Hence \(\mathcal{A}\restrict{\D} \subset \Sym^2(T\Omega)\restrict{\D}\) is a line bundle. After shrinking \(\D\), choose a nowhere vanishing section \(A \in \Gamma(\mathcal{A}\restrict{\D})\). We claim that \(A\) is nowhere indefinite. Indeed, if \(A_p\) were indefinite at some \(p \in \D\), then
    \begin{equation*}
        \{H \in \Sym^2(T_p^*\Omega) : \langle A_p, H\rangle_c = 0\} = \Ima \mathcal{B}_p
    \end{equation*}
    would contain a definite bilinear form. Hence there would exist a normal vector \(\nu\) such that \(II_p^\nu\) is definite, contradicting the saddle property. Thus \(A\) is nowhere indefinite. By \zcref{lem:compactly_supported_kernel}, there exists \(0 \neq \omega \in \Omega_c^1(D)\) such that
    \begin{equation*}
        L_A \omega = 0.
    \end{equation*}
    Since \(A\) generates the line bundle \(\mathcal{A}\restrict{\D}\), it follows that \(\sym(\nabla \omega) \in \Ima(\mathcal B\restrict{\D})\). With respect to the Euclidean metric, consider \((\ker (\mathcal{B})\restrict{\D})^\perp\). The restriction
    \begin{equation*}
        \mathcal{B}\restrict{(\ker (\mathcal{B})\restrict{\D})^\perp} : (\ker (\mathcal{B})\restrict{\D})^\perp \to \Ima(\mathcal{B} \restrict{\D})
    \end{equation*}
    is a vector bundle isomorphism. Hence there exists a unique smooth normal section \(\nu \in \Gamma((\ker (\mathcal{B})\restrict{\D})^\perp)\) satisfying \(\mathcal{B}(\nu) = \sym(\nabla \omega)\). Set \( V := \omega^\sharp\) and \(Y := Ds(V) + \nu\). Then
    \begin{equation*}
        \mathcal{E}(Y) = \sym(\nabla V^\flat) - \mathcal{B}(\nu) = \sym(\nabla \omega) - \sym(\nabla \omega) = 0.
    \end{equation*}
    Moreover, \(Y \neq 0\). Otherwise, its tangential and normal parts vanish, so \(V = 0\), hence \(\omega = 0\), a contradiction. Since \(\omega\) is compactly supported in \(\D\) and the inverse bundle map above is fibrewise linear, \(Y\) is compactly supported in \(D\).
\end{proof}

\subsubsection{Rank Properties}\label{subsubsec:rank_properties}

Let \(s : \Omega \to \R^d\) be a smooth immersion. For a smooth normal section \(\nu \in \Gamma(N_s \Omega)\), recall that the \emph{Weingarten map} \(W_\nu : T\Omega \to T\Omega\) is defined by
\begin{equation*}
    g(W_\nu X,Z) = II^\nu(X,Z) \quad \text{for all} \ X,Z \in \Gamma(T\Omega).
\end{equation*}
Equivalently,
\begin{equation*}
    Ds(W_\nu X) = - (D_X\nu)^\top.
\end{equation*}

\begin{lemma}\label{lem:rank_0-1_dense}
    Let \(s : \Omega \to \R^d\), \(d \geq 4\), be a smooth strictly saddle immersion, and let \(Y \in \Gamma(s^*(T\R^d))\) be an infinitesimal isometric deformation of \(s\). Then
    \begin{equation*}
        \mathcal{G} := \intset \{p \in \Omega : (DY)_p = 0\} \cup \{p \in \Omega : \rank (DY)_p = 2\}
    \end{equation*}
    is open and dense in \(\Omega\). If \(\Omega\) is connected and \(0 \neq Y \in \Gamma_c(s^*(T\R^d))\), then
    \begin{equation*}
        \{p \in \Omega: \rank (DY)_p = 2\} \neq \emptyset.
    \end{equation*}
\end{lemma}
\begin{proof}
    See \zcref{appendix:sec:auxiliary_proofs_for_infinitesimal_deformations}.
\end{proof}

\begin{example}\label{ex:rank_1_and_rank_2_infinitesimal_isometric_deformations}
  Let \(d \geq 4\) and consider the plane
  \begin{equation*}
      s : \R^2 \to \R^d, \ s(x_1, x_2) = (x_1, x_2, 0, \dots, 0).
  \end{equation*}
  This immersion is saddle but not strictly saddle. Its normal bundle is the bundle
  \begin{equation*}
      (N_s \R^2)_p = \spanset \{e_3, \dots, e_d\} \quad \text{for all} \ p \in \R^2,
  \end{equation*}
  and \(\mathcal B = 0\). Let \(D \subset \R^2\) be a disc and \(F=(f_3, \dots, f_d) \in C_c^\infty(D, \R^{d-2})\). We define the compactly supported normal section
  \begin{equation*}
      Y := \sum_{i = 3}^d f_i e_i.
  \end{equation*}
  Since \(Y\) is normal and \(\mathcal{B} = 0\), we find \(\mathcal{E}(Y) = - \mathcal{B}(Y) = 0\). Hence \(Y\) is an infinitesimal isometric deformation of \(s\). Since the normal frame \((e_3, \dots, e_d)\) is constant,
  \begin{equation*}
      DY_p(X) = \sum_{i = 3}^d Df_i(X)e_i \quad \text{for} \ X \in T_p\R^2.
  \end{equation*}
  Consequently, \(\rank (DY)_p = \rank (DF)_p\). In particular, for suitable choices of \(F\), the differential \(DY\) may have rank \(1\) or \(2\).
\end{example}

We conclude this section with a geometric consequence concerning the absence of rank-two infinitesimal isometric deformations.

\begin{lemma}\label{lem:existence_rank_2_infinitesimal_isometric_deformation}
    Let \(s : \Omega \to \R^4\) be a smooth saddle immersion and let \(U \subset \Omega\) be a non-empty open set on which \(\rank \mathcal B = 2\). Then \(s\) admits a compactly supported infinitesimal isometric deformation \(Y\) in \(U\) such that
    \begin{equation*}
        \rank (DY)_p = 2 \quad \text{for some} \ p \in U.
    \end{equation*}
\end{lemma}
\begin{proof}
    See \zcref{appendix:sec:auxiliary_proofs_for_infinitesimal_deformations}.
\end{proof}

\begin{proposition}\label{prop:surface_in_a_hyperplane}
   Let \(s : \Omega \to \R^4\) be a smooth saddle immersion and let \(U \subset \Omega\) be an open connected subset. Assume that there is no compactly supported infinitesimal isometric deformation \(Y \in \Gamma_c(s^*(T\R^4)\restrict{U})\) such that
   \begin{equation*}
       \rank (DY)_p = 2 \quad \text{at some} \ p \in U.
   \end{equation*}
   Then
   \begin{equation*}
       \mathcal{V} := \{p \in U : \rank (\mathcal{B}_p) = 1\}
   \end{equation*}
   is open and dense in \(U\). Moreover, every connected component of \(\mathcal{V}\) is mapped by \(s\) into an affine hyperplane of \(\R^4\).
\end{proposition}
\begin{proof}
    Since \(s\) is saddle, we have 
    \begin{equation*}
        \rank \mathcal{B}\restrict{U} \leq 2,
    \end{equation*}
    since otherwise \(\Ima \mathcal{B}_p = \Sym^2(T_p^*\Omega)\), which contains definite forms. If \(\rank \mathcal{B}_p = 2\) at some \(p \in U\), then, after restricting \(U\) to a small neighbourhood of \(p\), \(\rank \mathcal{B}\restrict{U} = 2\). Therefore \zcref{lem:existence_rank_2_infinitesimal_isometric_deformation} yields an infinitesimal isometric deformation \(Y\) with \(\rank (DY)_p\) at some \(p \in U\) contradicting the hypothesis. Hence \(\rank \mathcal{B}\restrict{U} \leq 1\), so \(\mathcal{V} = \{p \in U : \mathcal{B}_p \neq 0\}\) is an open set.

    If \(\mathcal{V}\) were not dense, there would exist an open disc \(D \subset U\) on which \(\mathcal{B} = 0\). Choose an orthonormal frame \((N_1,N_2)\) of \(N_s\Omega\restrict{D}\), a point \(p \in D\), and \(\varphi, \psi \in C_c^\infty(D)\) such that
    \begin{equation*}
        \varphi(p) = \psi(p) = 0 \quad \text{and} \quad D\varphi_p, \ D\psi_p \ \text{are linearly independent.}
    \end{equation*}
    Then
    \begin{equation*}
        Y := \varphi N_1 + \psi N_2
    \end{equation*}
    satisfies \(\mathcal{E}(Y) = - \mathcal{B}(Y) = 0\), so \(Y\) is an infinitesimal isometric deformation. Moreover, we have
    \begin{equation*}
        (DY)_p(X) = D\varphi_p(X)N_1(p) + D\psi_p(X)N_2(p),
    \end{equation*}
    and therefore \(\rank (DY)_p = 2\), a contradiction. Thus \(\mathcal{V}\) is dense.

    Let \(W\) be a connected component of \(\mathcal{V}\). Since \(\rank \mathcal{B}\restrict{W} = 1\),
    \begin{equation*}
        K := \ker (\mathcal B \restrict{W})
    \end{equation*}
    is a smooth line bundle. Fix \(p \in W\), choose an open neighbourhood \(O \subset W\) of \(p\), and let \(\nu \in \Gamma(K \restrict{O})\) be a smooth unit section. Then \(II^\nu = 0\), so by the Weingarten equation \(D_X \nu\) is normal for \(X \in \Gamma(T\Omega)\), and using \(|\nu| = 1\), we find
    \begin{equation*}
        0 = X\langle\nu,\nu\rangle = 2\langle D_X \nu, \nu\rangle,
    \end{equation*}
    so \(D_X \nu\) is orthogonal to \(\nu\).
    
    Now we show \(D \nu = 0\). If \(D\nu \neq 0\) at some \(p\), choose linearly independent \(X_1, X_2 \in T_p\Omega\) with \(D_{X_1}\nu \neq 0\) and \(\varphi \in C_c^\infty(O)\) such that \(\varphi(p) = 1\), \(D\varphi_p(X_1) = 0\) and \(D\varphi_p(X_2) \neq 0\). Then \(Y := \varphi \nu\) is infinitesimally isometric, and
    \begin{align*}
        (DY)_p(X_1) &= D\varphi_p(X_1)\nu(p) + \varphi(p) D_{X_1}\nu = D_{X_1} \nu\\
        (DY)_p(X_2) &= D\varphi_p(X_2)\nu(p) + D_{X_2}\nu.
    \end{align*}
    These vectors are linearly independent because
    \begin{equation*}
        \langle (DY)_p(X_1), \nu(p)\rangle = 0 \quad \text{and} \quad \langle (DY)_p(X_2), \nu(p)\rangle = D \varphi_p(X_2) \| \nu(p)\| \neq 0, 
    \end{equation*}
    again contradicting the hypothesis. Hence \(D\nu = 0\), so \(\nu\) is locally constant and, since \(W\) is connected, constant on \(W\). Thus there exists a fixed \(0 \neq \nu \in \R^4\) normal to \(s\) along \(W\). Hence
    \begin{equation*}
        X \langle s, \nu\rangle = \langle Ds(X), \nu\rangle = 0 \quad \text{for every} \ X \in \Gamma(TW).
    \end{equation*}
    Therefore \(\langle s, \nu\rangle\) is constant on \(W\), and \(s(W)\) is contained in an affine hyperplane of \(\R^4\).
\end{proof}

Compare this result with 
\begin{proposition}[{\cite[Proposition 1.2]{petruninstadler2019metric}}]
    Every smooth strictly saddle surface in \(\R^3\) is locally metric-minimising.
\end{proposition}

\section{The Density Theorem}\label{sec:the_density_theorem}

We now prove the density theorem \zcref{introduction:thm:density_theorem} by adding a sequence of small compactly supported infinitesimal isometric deformations to the original surface.

\begin{lemma}\label{lem:strictly_saddle_is_an_open_condition}
    Let \(\Omega \subset \R^2\) be an open disc and \(s_0 : \overline{\Omega} \to \R^4\) be a smooth, strictly saddle embedding on \(\overline{\Omega}\). Then we find \(\varepsilon_* > 0\) such that every smooth map \(s:\overline{\Omega} \to \R^4\) with \(\|s - s_0\|_{C^2(\overline{\Omega})} < \varepsilon_*\) is a smooth, strictly saddle embedding on \(\overline{\Omega}\) as well.
\end{lemma}
\begin{proof}
    See \zcref{appendix:sec:auxiliary_proof_for_the_density_theorem}.
\end{proof}

\begin{theorem}\label{thm:density_theorem}
    Let \(\Omega \subset \C\) be an open disc and let \(s_0 : \overline{\Omega} \to \C^2\) be a smooth, strictly saddle embedding. Then, for every integer \(a \geq 0\) and every \(\varepsilon > 0\), there exists a smooth, strictly saddle embedding \(s_\infty:\overline{\Omega} \to \C^2\) such that
    \begin{equation}\label{eq:thm:density_theorem:1}
        \|s_\infty-s_0\|_{C^a(\overline{\Omega})} < \varepsilon,
    \end{equation}
    and \(s_\infty\) is nowhere locally metric-minimising in \(\Omega\).
\end{theorem}
\begin{proof}
    It suffices to prove the statement for \(a\geq2\), since the cases \(a=0,1\) follow from the case \(a=2\). We may further assume that \(\varepsilon\leq\varepsilon_*\), where \(\varepsilon_*\) is the constant from \zcref{lem:strictly_saddle_is_an_open_condition}. Pick a countable basis \(\{U_n\}_{n\in\N}\) for the topology of \(\Omega\), consisting of open discs \(U_n \Subset \Omega\).
    
    \emph{Step 1:} We are going to construct by induction:
    \begin{enumerate}[label=(\alph*)]
        \item smooth strictly saddle embeddings \(s_n : \overline{\Omega} \to \C^2\)
        \item non-trivial compactly supported infinitesimal isometric deformations \(Y_n\)
        \item numbers \(t_n > 0\) and open discs \(D_n \Subset U_n\) with
        \begin{equation*}
            s_n = s_{n-1} + t_nY_n.
        \end{equation*}
    \end{enumerate}
    
    Assume that we have already constructed \(s_0, \dots, s_{n-1}\) and \(Y_1, \dots, Y_{n-1}\). For \(j < n\), set
    \begin{equation*}
        Z_j := \intset \{p \in \Omega : DY_j(p) = 0\}, \quad R_j := \{p \in \Omega :  \rank DY_j(p) = 2\} \quad \text{and} \quad \mathcal{G}_j := Z_j \cup R_j.
    \end{equation*}
    The set \(\mathcal{G}_j\) is open and dense in \(\Omega\) according to \zcref{lem:rank_0-1_dense}. Pick a point \(p_n\) in the non-empty open set \(U_n \cap \bigcap_{j < n} \mathcal{G}_j\). For every \(j < n\), we have either \(p_n \in Z_j\) or \(p_n \in R_j\). Therefore, we find an open disc \(V_n\) around \(p_n\) with \(V_n \Subset U_n\) satisfying
    \begin{equation*}
        V_n \Subset Z_j \quad \text{or} \quad V_n \Subset R_j \quad \text{for every} \ j < n.
    \end{equation*}
    We find a non-trivial infinitesimal isometric deformation \(Y_n \in C_c^\infty(V_n,\C^2)\) of \(s_{n-1}\) by \zcref{prop:existence_infinitesimal_deformation}. Pick a disc \(D_n\) with \(\supp Y_n \Subset D_n \Subset V_n\) and define, for every index \(j < n\) with \(V_n \Subset R_j\)
    \begin{equation*}
        \mu_j(x) := \inf_{|\xi| = 1} |DY_j(x)\xi| \quad \text{and} \quad M_{j,n} := \sup_{x \in \overline{V_n}} \frac{\|DY_n(x)\|_{op}}{\mu_j(x)}.
    \end{equation*}
    In this case, \(DY_j\) has rank two on \(\overline{V}_n\), so \(\mu_j\) is continuous and strictly positive there. 
    
    Pick \(t_n\) small enough such that
    \begin{equation*}
        \|t_nY_n\|_{C^{n+a}(\overline{\Omega})} < 2^{-n-2}\varepsilon \quad \text{and} \quad t_n M_{j,n} < 2^{-n-2} t_j \ \text{for all} \ j < n \ \text{with} \ V_n \Subset R_j.
    \end{equation*}
    We now define
    \begin{equation*}
        s_n := s_{n-1} + t_n Y_n.
    \end{equation*}
    By \zcref{lem:strictly_saddle_is_an_open_condition}, \(s_n\) is a smooth, strictly saddle embedding since
    \begin{equation*}
        \|s_n - s_0\|_{C^2(\overline{\Omega})} \leq \sum_{i=1}^n \|s_i - s_{i-1}\|_{C^2(\overline{\Omega})} \leq \sum_{i=1}^n \|t_iY_i\|_{C^2(\overline{\Omega})} < \varepsilon,
    \end{equation*}
    since the geometric series gives \(\sum_{i=1}^\infty 2^{-i-2} = \frac{1}{4}\).  
    
    \emph{Step 2:} Now we prove that \(s_\infty\) is a smooth, strictly saddle embedding. It is easy to see that \(\{s_n\}_{n \in \N}\) is a Cauchy sequence in \(C^m\) for every \(m \in \N_0\). Indeed, for fixed \(m \in \N_0\) and arbitrary integers \(k \geq l \geq m\), we have
    \begin{equation*}
        \begin{split}
            \|s_k - s_l\|_{C^{m+2}(\overline{\Omega})} &\leq \sum_{i = l+1}^k \|s_i - s_{i-1}\|_{C^{m+2}(\overline{\Omega})} \leq \sum_{i=l+1}^k \|s_i - s_{i-1}\|_{C^{m+a}(\overline{\Omega})}\\
            &\leq \sum_{i=l+1}^k \|s_i - s_{i-1}\|_{C^{i+a}(\overline{\Omega})} < \sum_{i=l+1}^\infty 2^{-i-2}\varepsilon \xrightarrow{l \to \infty} 0.
        \end{split}
    \end{equation*}
    Thus we find a smooth map \(s_\infty:\overline{\Omega} \to \C^2\) such that \(s_n \to s_\infty\) in \(C^m\) for every \(m\). By the same computation, we find
    \begin{equation*}
        \|s_\infty - s_0\|_{C^{a}(\overline{\Omega})} \leq \sum_{i=1}^\infty \|t_i Y_i\|_{C^{a}(\overline{\Omega})} \leq \sum_{i=1}^\infty \|t_i Y_i\|_{C^{i+a}(\overline{\Omega})} \leq  \varepsilon \sum_{i=1}^\infty 2^{-i-2} = \frac{\varepsilon}{4} < \varepsilon.
    \end{equation*}
    Consequently, the map \(s_\infty\) is a smooth, strictly saddle embedding by \zcref{lem:strictly_saddle_is_an_open_condition}, and we have proved \zcref{eq:thm:density_theorem:1}.

    \emph{Step 3:} Here we construct a local competitor of \(s_\infty\) to show that \(s_\infty\) is not locally metric-minimising. We fix \(j \in \N\) and set
    \begin{equation*}
        B_j := s_\infty - s_j = \sum_{n = j+1}^\infty t_n Y_n.
    \end{equation*}
    Define the map
    \begin{equation*}
        q_j : \Omega \to \C^2, \ q_j := s_{j-1} + B_j,
    \end{equation*}
    so that \(s_\infty = q_j + t_jY_j\). The map \(q_j\) will serve as a competitor. It follows from \(\supp Y_j \Subset D_j\) that
    \begin{equation*}
        q_j \restrict{\Omega \setminus D_j} = s_\infty \restrict{\Omega \setminus D_j}.
    \end{equation*}
    For a tangent vector \(X \in T_p \Omega\), we obtain
    \begin{equation*}
        \begin{split}
            (g_{s_\infty} - g_{q_j})(X,X) &= 2t_j\langle Dq_j(X),DY_j(X)\rangle + t_j^2 |DY_j(X)|^2\\
            &= 2t_j \langle Ds_{j-1}(X), DY_j(X)\rangle + 2t_j \langle DB_j(X), DY_j(X)\rangle + t_j^2 |DY_j(X)|^2\\
            &= 2t_j \langle DB_j(X), DY_j(X)\rangle + t_j^2 |DY_j(X)|^2.
        \end{split}
    \end{equation*}
    In the last equality, we used that the first term vanishes since \(Y_j\) is an infinitesimal isometric deformation of \(s_{j-1}\).
    
    \emph{Case 1:} \(\rank DY_j(p) = 0\). Then we immediately get \((g_{s_\infty} - g_{q_j})(X,X) = 0\).
    
    \emph{Case 2:} \(\rank DY_j(p) = 1\). In this case we have \(p \not \in \mathcal{G}_j\), so \(p \not\in \overline{V}_n\) for all \(n > j\) by construction. Since \(\supp Y_n \Subset V_n\), the deformation \(Y_n\) vanishes in a neighbourhood of \(p\) for every \(n > j\), so 
    \begin{equation}\label{eq:thm:existence_not_locally_metric-minimising_saddle_surface:1}
        (g_{s_\infty} - g_{q_j})(X,X) = t_j^2 |DY_j(X)|^2 \geq 0.
    \end{equation}
    
    \emph{Case 3:} \(\rank DY_j(p) = 2\). Then we have \(|DY_j(X)| \geq \mu_j(p)|X|\). First suppose that \(DY_n(p) = 0\) for every \(n > j\). Then we are in the situation of \zcref{eq:thm:existence_not_locally_metric-minimising_saddle_surface:1}. Otherwise, let \(n > j\) be any index such that \(DY_n(p) \neq 0\). Then \(p \in V_n\). Since \(p \in R_j\), we infer \(\overline{V}_n \subset R_j\) by construction, and
    \begin{equation*}
        t_n\|DY_n(p)\|_{op} \leq t_nM_{j,n} \mu_j(p) \leq 2^{-n-2}t_j \mu_j(p).
    \end{equation*}
    Hence,
    \begin{equation*}
        \|DB_j(p)\|_{op} \leq \sum_{n=j+1}^\infty t_n\|DY_n(p)\|_{op} \leq t_j \mu_j(p) \sum_{n=j+1}^\infty 2^{-n-2} \leq \frac{1}{4}t_j\mu_j(p),
    \end{equation*}
    and therefore
    \begin{equation*}
        2t_j|\langle DB_j(X), DY_j(X)\rangle| \leq \frac{1}{2}t_j^2|DY_j(X)|^2.
    \end{equation*}
    It follows that
    \begin{equation*}
        (g_{s_\infty} - g_{q_j})(X,X) \geq \frac{1}{2}t_j^2|DY_j(X)|^2.
    \end{equation*}
    The expression is strictly positive for \(X \neq 0\). The three cases show
    \begin{equation*}
        g_{s_\infty} \geq g_{q_j} \quad \text{on} \ \Omega
    \end{equation*}
    with strict inequality as quadratic forms whenever \(\rank DY_j(p) = 2\).

    \emph{Step 4:} Now we prove that \(s_\infty\) is nowhere locally metric-minimising. Fix \(p \in \Omega\), and let \(\D \Subset \Omega\) be an arbitrary disc centred at \(p\). Since \(\{U_n\}_{n \in \N}\) is a basis of the topology of \(\Omega\), we find \(n \in \N\) with \(D_n \Subset U_n \Subset \D\). Since \(Y_n\) is non-trivial and \(\Omega\) is connected, \zcref{lem:rank_0-1_dense} implies that \(DY_n\) has rank \(2\) on a non-empty open subset of \(D_n\). By Step 3,
    \begin{equation*}
        q_n \restrict{\partial\D} = s_\infty\restrict{\partial\D} \quad \text{and} \quad g_{s_\infty} \geq g_{q_n} \ \text{on} \ \D
    \end{equation*}
    with strict inequality as quadratic forms on a non-empty open subset of \(D_n\). Consequently,
    \begin{equation*}
        \ell(s_\infty \circ \gamma) \geq \ell(q_n \circ \gamma) \quad \text{for every Lipschitz curve} \ \gamma \ \text{in} \ \D,
    \end{equation*}
    and the inequality is strict for some smooth curve contained in the set where \(\rank DY_n = 2\). Thus \(q_n\) is the desired competitor for \(s_\infty\). Since \(\D\) was arbitrary, \(s_\infty\) is nowhere locally metric-minimising.
\end{proof}

\section*{Declarations}

\paragraph{Funding.} This research received no external funding.

\paragraph{Conflict of Interest.} I do not have any competing interest.

\paragraph{Usage of Artificial Intelligence.}
During the preparation of this work, I used OpenAI's ChatGPT (specifically versions GPT-5.5 and GPT-5.6) to assist with routine computations, review and simplify my own arguments and computations, and proofread the manuscript. Additionally, based on an initial unsuccessful approach provided as input, the tool suggested the use of infinitesimal isometric deformations (\zcref{def:infinitesimal_isometric_deformation}), which successfully circumvented the problem. I thoroughly reviewed and verified all AI-generated outputs and confirm sole responsibility for the final content of the article.

\paragraph{Data and Code Availability.} Data sharing is not applicable to this article as no datasets or computer code were generated or analysed during the research.

\section*{Appendix}\label{sec:appendix}
\appendix

\section[Proofs for Saddle and Metric-Minimising Surfaces]{Proofs for \zcref{sec:saddle_and_metric-minimising_surfaces}}\label{appendix:sec:proofs_saddle_and_metric_minimising_surfaces}

\begin{proof}[Proof of \zcref{prop:locally_metric-minimising_saddle}]
    Assume, to the contrary, that \(s\) is not saddle at \(p\). Choose a disc \(\D\) centred at \(p\) such that \(s\restrict{\D}\) is metric-minimising. Then there exists an affine hyperplane
    \begin{equation*}
        \Pi =\{x \in \R^d : \langle a, x\rangle = c\}
    \end{equation*}
    with \(a \in \Sp^{d-1}\) and a connected component \(C\) of \(\D \setminus s^{-1}(\Pi)\) such that \(C \cap \partial \D = \emptyset\). Set \(h(x) := \langle a, x\rangle - c\). After replacing \(a\) by \(-a\) we may assume that \(h > 0\) on \(C\). Moreover, \(h \equiv 0\) on \(\partial C\). Now we consider the orthogonal projection
    \begin{equation*}
        P : \R^d \to \Pi, \ P(x) = x - h(x)a,
    \end{equation*}
    and set
    \begin{equation*}
        \tilde{s}(q) = \begin{cases}
            P(s(q)) = s(q)-h(q)a, & q \in C\\
            s(q), & q \notin C.
        \end{cases}
    \end{equation*}
    We need to check that \(\tilde{s}\) is Lipschitz continuous. The only non-trivial part is the case \(x \in C\) but \(y \notin C\). Pick a point in the segment \(z \in [x,y]\) with \(z \in \partial C\). We find
    \begin{equation*}
        \begin{split}
            |\tilde{s}(x)-\tilde{s}(y)| &\leq |\tilde{s}(x) - \tilde{s}(z)| + |\tilde{s}(z) - \tilde{s}(y)|
            = |P(s(x)) - P(s(z))| + |\tilde{s}(z) - \tilde{s}(y)|\\
            & \leq \Lip(s)|x-z| + \Lip(s)|z-y| \leq \Lip(s)|x-y|,
        \end{split}
    \end{equation*}
    and, in particular \(\Lip(\tilde{s}) \leq \Lip(s)\).

    Let \(\gamma : [0,1] \to \D\) be an arbitary Lipschitz curve, set \(u = s \circ \gamma\), \(\tilde{u} = \tilde{s} \circ \gamma\) and \(E := \gamma^{-1}(C)\). Both functions \(u\) and \(\tilde{u}\) are Lipschitz, and therefore almost everywhere differentiable. Since \(E\) is open, we may write locally around every \(t \in E\) \(\tilde{u} = P \circ u\), so we have for almost every \(t \in E\) the identity
    \begin{equation*}
        \tilde{u}'(t) = DP(u'(t)) = u'(t) - \langle a, u'(t)\rangle a.
    \end{equation*}
   In particular, we get \(|\tilde{u}'(t)| \leq |u'(t)|\). On the complement \(F = [0,1] \setminus E\) we have \(u = v\), hence \(u' = v'\) almost everywhere. So we obtain \(|v'| \leq |u'|\) almost everywhere on \([0,1]\), and, in particular, 
    \begin{equation}\label{eq:prop:locally_metric-minimising_saddle:1}
        \ell(\tilde{s} \circ \gamma) = \int_0^1 |v'(t)| dt \leq \int_0^1 |u'(t)|dt = \ell(s \circ \gamma).
    \end{equation}
    It remains to find a Lipschitz curve \(\gamma\) with strict inequality in \zcref{eq:prop:locally_metric-minimising_saddle:1}. Choose two points \(q_1, q_2 \in C\) with \(h(q_0) \neq h(q_1)\) and join them by a Lipschitz curve \(\gamma: [0,1] \to C\). Since the function \(h \circ \gamma\) is not constant, we find \((h \circ \gamma)' \neq 0\) on a set of positive measure. We thereby infer
    \begin{equation*}
        |(\tilde{s} \circ \gamma)'|^2 = |(s \circ \gamma)'|^2 - |(h \circ \gamma)'|^2,
    \end{equation*}
    so we conclude \(\ell(\tilde{s} \circ \gamma) < \ell(s \circ \gamma)\), a contradiction!
\end{proof}

\section[Auxiliary Proofs for Infinitesimal Deformations]{Auxiliary Proofs for \zcref{sec:infinitesimal_deformations}}\label{appendix:sec:auxiliary_proofs_for_infinitesimal_deformations}

\begin{proof}[Proof of \zcref{lem:twisting_form_graph_coordinates}]
    We have
    \begin{equation*}
        \langle A,II^\nu\rangle_c = A^{ij}II_{ij}^\nu = \langle A^{ij}II_{ij}, \nu\rangle = 0 \quad \text{for all} \ \nu \in (N_s\Omega)_p,
    \end{equation*}
    hence \(A^{ij}II_{ij} = 0\). On the other hand, we find \(A^{ij}s_{ij} = 0\). Indeed, let us consider \(k = 1, \dots, d-2\) the normal frame
    \begin{equation*}
        N_k := (-\partial_1 u^k, - \partial_2 u^k, 0, \dots, \underbrace{1}_k, \dots, 0).
    \end{equation*}
    We get
    \begin{equation*}
        II_{ij}^{N_k} = \langle II_{ij}, N_k\rangle = \langle s_{ij}, N_k\rangle = \partial_{ij}u^k.
    \end{equation*}
    It follows that \(\mathcal{B}(N_k) = D^2 u^k\), and we obtain \(0 = \langle A, D^2u^k\rangle_c = A^{ij}\partial_{ij}u^k\), so \(A^{ij}s_{ij} = 0\). The Gauss formula \cite[Theorem 8.2]{lee2018introduction} yields \(s_{ij} = \Gamma_{ij}^k s_k + II_{ij}\), and multiplying by \((A^{ij})\) results in
    \begin{equation}\label{eq:lem:twisting_form_graph_coordinates:1}
        A^{ij}\Gamma_{ij}^k = 0.
    \end{equation}
    Now we compute
    \begin{equation*}
        (\divg A)^j = \nabla_i A^{ij} = \partial_i A^{ij} + \Gamma_{ik}^i A^{kj} + \Gamma_{ik}^j A^{ik}.
    \end{equation*}
    The last term vanishes by \zcref{eq:lem:twisting_form_graph_coordinates:1}, and furthermore we have \(\Gamma_{ik}^i =\partial_k \log \sqrt{\det g}\). Hence,
    \begin{equation*}
        \divg A = \divg_0 A + \Phi_A (d\log \sqrt{\det g}),
    \end{equation*}
    and applying \(-(\Phi_A)^{-1}\) gives
    \begin{equation*}
        \alpha_A = \alpha - d\log \sqrt{\det g}.
    \end{equation*}
    Taking derivatives yields \(\Theta = d\alpha_A = d\alpha\). Finally, if \(\alpha = \alpha_1 dx + \alpha_2 dy\), then
    \begin{equation*}
        d\alpha = (\partial_x \alpha_2 - \partial_y \alpha_1) dx \wedge dy.
    \end{equation*}
\end{proof}

\begin{lemma}\label{appendix:lem:semidefinite_reduction}
    Let \(\D\) be a disc and \(A \in \Gamma(\Sym^2(T\D))\) be nowhere vanishing and nowhere indefinite. After replacing \(A\) by \(-A\) if necessary, \(A\) is positive semidefinite on \(\D\).
\end{lemma}
\begin{proof}
    Since \(A\) is symmetric, nowhere vanishing, and nowhere indefinite, we have
    \begin{equation*}
        \trace_g A = \trace(\Phi_A \circ \Psi_g) \neq 0
    \end{equation*}
    at every point. Hence, \(\trace_g A\) has constant sign. After replacing \(A\) by \(-A\) if necessary, we may assume \(\trace_g A > 0\). Since \(A\) is nowhere indefinite, it follows that \(A\) is positive semidefinite.
\end{proof}

\begin{lemma}\label{appendix:lem:existence_not_locally_metric-minimising_saddle_surface_higher_dimension}
    For every \(d \geq 4\), there exist an open disc \(\Omega \subset \C\) and a smooth saddle embedding \(s:\overline{\Omega} \to \R^d\) that is nowhere locally metric-minimising in \(\Omega\).
\end{lemma}
\begin{proof}
    Let \(\tilde{s}:\overline{\Omega}\to\R^4\) be the twisted local model \eqref{eq:subsubsec:an_application_to_the_twisted_local_model:1}, where \(\overline{\Omega}\) is a small disc around the origin on which \(\tilde{s}\) is saddle but nowhere locally metric-minimising. Set
    \begin{equation*}
        \iota:\R^4 \hookrightarrow \R^d, \ \iota(x) = (x,0,\dots,0)
    \end{equation*}
    to be the canonical isometric embedding. 
    
    We define \(s := \iota \circ \tilde{s}\) and check that \(s\) is the desired map. Since \(\iota\) is a smooth embedding, \(s\) is a smooth embedding as well. It remains to verify the saddle property. Since \(s(\Omega) \subset \iota(\R^4)\), we obtain for every hyperplane \(\Pi \subset \R^d\)
    \begin{equation*}
        s^{-1}(\Pi) = s^{-1}(\Pi \cap \iota(\R^4)) = \tilde{s}^{-1}(\iota^{-1}(\Pi \cap \iota(\R^4))).
    \end{equation*}
    The set \(\iota^{-1}(\Pi \cap \iota(\R^4))\) is either empty, the whole space \(\R^4\), or an affine hyperplane in \(\R^4\). In the first two cases, the statement is clear, and in the third it follows from the saddle property of \(\tilde{s}\). Since \(\tilde{s}\) is nowhere locally metric-minimising and \(\iota\) is an isometric embedding, every local competitor for \(\tilde{s}\) yields, after composition with \(\iota\), a local competitor for \(s\). Thus \(s\) is nowhere locally metric-minimising as well.
\end{proof}

\begin{lemma}\label{appendix:lem:positive_definite_case}
    Let \(D \subset \Omega\) be a disc and let \(A \in \Gamma(\Sym^2 TD)\) be positive definite. Then every non-empty open subset \(U \subset D\) contains a non-zero \(1\)-form \(\omega \in \Omega_c^1(U)\) such that
    \begin{equation*}
        L_A \omega = 0.
    \end{equation*}
\end{lemma}
\begin{proof}
    As in the proof of \zcref{lem:contraction_operator_surjective}, for \(X = \Phi_A(\omega)\) and \(\eta := *X^\flat\), the equation \(L_A \omega = 0\) is equivalent to
    \begin{equation}\label{eq:appendix:lem:positive_definite_case:1}
        d\eta + \alpha_A \wedge \eta = 0.
    \end{equation}
    We distinguish two cases.

    \emph{Case 1:} \(d\alpha_A \equiv 0\) on \(U\). Choose a disc \(B \Subset U\). By the Poincar\'e Lemma \cite[Theorem 11.49]{lee2012introduction}, there exists \(u \in C^\infty(B)\) with \(\alpha_A = du\). For a non-constant \(\varphi \in C_c^\infty(B)\), set \(\eta = e^{-u}d\varphi\). Then \(\eta \in \Omega_c^1(B)\), \(\eta \not\equiv 0\), and
    \begin{equation*}
        d\eta + \alpha_A \wedge \eta = d(e^{-u}d\varphi) + du\wedge e^{-u}d\varphi = 0.
    \end{equation*}

    \emph{Case 2:} Now suppose \(d\alpha_A \not\equiv 0\) on \(U\). After shrinking, choose a coordinate rectangle \(I \times J \Subset U\) on which \(d\alpha_A\) is nowhere vanishing. We construct a non-zero local solution of \(L_A \omega = 0\) on \(I \times J\). Write
    \begin{equation*}
        \alpha_A = a(x,y) dx + b(x,y)dy \quad \text{and} \quad \eta = p(x,y) dx + q(x,y) dy.
    \end{equation*}
    Then \zcref{eq:appendix:lem:positive_definite_case:1} reduces to
    \begin{equation*}
        \partial_x q - \partial_y p + aq - bp = 0.
    \end{equation*}
    We may set \(q = 0\) and the equation reduces to
    \begin{equation*}
        \partial_y p + bp = 0.
    \end{equation*}
    Now we fix \(y_0\) and require the initial solution \(p(x,y_0) = 1\). Then
    \begin{equation*}
        p(x,y) = \exp\left(- \int_{y_0}^y b(x,t) dt\right) > 0,
    \end{equation*}
    so \(\tilde{\eta} := pdx\) is a nowhere vanishing local section. Let \(\tilde{\omega}\) denote the corresponding nowhere vanishing one-form such that \(L_A \tilde{\omega} = 0\). Now we use this solution to construct a compactly supported one form. Choose discs \(V \Subset W \Subset I \times J\) and a cut-off function \(\chi \in C_c^\infty(W)\) such that \(\chi \equiv 1\) on \(V\). Set \(f := L_A(\chi\tilde{\omega})\). Let \(Q_A\) be the first-order differential right inverse of \(L_A\) from \zcref{lem:contraction_operator_surjective}, and set
    \begin{equation*}
        \omega := \chi\tilde{\omega} - Q_A(f).
    \end{equation*}
    Then
    \begin{equation*}
        L_A\omega = L_A(\chi\tilde{\omega}) - L_A(R(f)) = f - f = 0.
    \end{equation*}
    Since \(L_A\tilde{\omega} = 0\), we have \(\supp f \subset \supp d\chi\). Moreover, since \(R\) is a differential operator, it does not enlarge supports, and hence \(\supp R(f) \subset \supp f \subset \supp d\chi\). Since \(\chi \equiv 1\) on \(V\), we get \(d\chi = 0\) there, and therefore \(Q_A(f) = 0\) on \(V\). Thus \(\omega = \tilde{\omega}\) on \(V\). Because \(\tilde{\omega}\) is nowhere vanishing, it follows that \(\omega \not \equiv 0\).
\end{proof}

\begin{remark}\label{appendix:rem:alternative_proof_positive_definite_case}
    Alternatively, the result may be obtained from the general results on compactly supported solutions of underdetermined elliptic equations in \cite{delay2012smooth}.   
\end{remark}

\begin{lemma}\label{appendix:lem:rank_one_case}
    Let \(D \subset \Omega\) be an open disc and \(A \in \Gamma(\Sym^2 TD)\) be positive semidefinite of constant rank one. Then there exists \(p \in D\) and a non-zero \(\omega \in \Omega_c^1(D)\) such that
    \begin{equation*}
        L_A \omega = 0.
    \end{equation*}
    Moreover, writing \(S := \sym(\nabla \omega)\), one may arrange that
    \begin{equation*}
        \omega_p = 0, \quad S(p) = 0 \quad \text{and} \quad \rank (\nabla S)_p = 2.
    \end{equation*}
\end{lemma}
\begin{proof}
    After shrinking to a sufficiently small disc \(U \subset D\), we may write \(A = \lambda X \otimes X\), where \(\lambda > 0\) and \(X\) is a smooth unit vector field. We complete \(X\) to an orthonormal frame \((X,Z)\) with dual coframe \((\theta, \eta)\). Since \(X\) has unit length, there exists a smooth function \(\mu\) such that \(\nabla_X X = \mu Z\). By writing \(\omega = f\theta + h\eta\), we obtain
    \begin{equation*}
        \begin{split}
            L_A \omega &= \lambda \sym(\nabla \omega)(X,X) = \lambda(\nabla_X \omega)(X) = \lambda X(\omega(X)) - \lambda\omega(\nabla_XX) = \lambda Xf - \lambda \omega(\mu Z)\\
            &= \lambda Xf - \lambda \mu h
        \end{split}
    \end{equation*}
    Thus, it suffices to construct compactly supported smooth functions \(f\), \(h\) with \(Xf = \mu h\). We shall arrange that, at some point \(p \in U\),
    \begin{equation}\label{eq:lem:appendix:lem:rank_one_case:1}
        \omega = \eta\left(x^2 + \frac{y^2}{2}\right) + O(r^3),
    \end{equation}
    where \((x,y)\) are normal coordinates centred at \(p\) such that \(X_p = \partial_x \restrict{p}\) and \(Z_p = \partial_y\restrict{p}\).
    
    \emph{Case 1:} \(\mu \equiv 0\) on \(U\). Choose \(p \in U\), normal coordinates as above, and a cutoff function \(\chi \in C_c^\infty(U)\) equal to \(1\) near \(p\). Set
    \begin{equation*}
        f = 0 \quad \text{and} \quad h = \chi\left(x^2 + \frac{y^2}{2}\right).
    \end{equation*}
    Then \(Xf = \mu h = 0\), and \zcref{eq:lem:appendix:lem:rank_one_case:1} holds.

    \emph{Case 2:} \(\mu\) is not identically zero in \(U\). Choose \(p \in U\) with \(\mu(p) \neq 0\), and a disc \(V \subset U\) containing \(p\) on which \(\mu\) is nowhere zero. Let \((x,y)\) be normal coordinates as above, and choose a function \(f \in C_c^\infty(V)\) with
    \begin{equation*}
        f = O(r^3) \quad \text{and} \quad \partial_x f = \mu(p)\left(x^2 + \frac{y^2}{2}\right) + O(r^3).
    \end{equation*}
    (Integrate the quadratic polynomial w. r. t. \(x\) and multiply the corresponding cubic polynomial by a cutoff function.) Now define \(h := \frac{Xf}{\mu}\). Then \(f,h \in C_c^\infty(V)\), and \(Xf - \mu h = 0\). Moreover, we have \(X = \partial_x + O(r)\), whence \(Xf = \partial_x f + O(r^3)\). Together with \(\mu = \mu(p) + O(r)\), we therefore obtain
    \begin{equation*}
        h = \frac{Xf}{\mu} = \left(x^2 + \frac{y^2}{2}\right) + O(r^3).
    \end{equation*}
    Consequently, \zcref{eq:lem:appendix:lem:rank_one_case:1} follows as required.

    In either case, extending \(\omega\) by zero outside its compact support yields a smooth, compactly supported \(1\)-form on \(D\) satisfying \(L_A\omega = 0\). It remains to verify the additional conditions on \(S = \sym(\nabla \omega)\). By construction, \(f\), \(h\) and their first derivatives vanish at \(p\), hence \(\omega_p = 0\) and \((\nabla \omega)_p = 0\), and, in particular \(S(p)\) = 0. The first derivatives of \(S\) at \(p\) are determined by the second-order Taylor polynomial of \(\omega\), so we find using \zcref{eq:lem:appendix:lem:rank_one_case:1}
    \begin{equation*}
        (\nabla_X S)_p = \theta_p \otimes \eta_p + \eta_p \otimes \theta_p \quad \text{and} \quad (\nabla_Z S)_p = \eta_p \otimes \eta_p.
    \end{equation*}
    These two symmetric tensors are linearly independent, hence \(\rank(\nabla S)_p = 2\) holds as well.
\end{proof}

\begin{proof}[Proof of \zcref{lem:compactly_supported_kernel}]
   Fix an arbitrarily small disc \(B \subset D\). By \zcref{appendix:lem:semidefinite_reduction} we may assume that \(A\) is positive semidefinite on \(B\). We distinguish two cases.

   \emph{Case 1:} \(A\) is positive definite at some point of \(B\). After shrinking \(B\) around this point, we may assume that \(A\) is positive definite on \(B\). By \zcref{appendix:lem:positive_definite_case} there exists a non-zero \(1\)-form \(\omega \in \Omega_c^1(B)\) satisfying \(L_A \omega = 0\).

   \emph{Case 2:} \(A\) is nowhere positive definite on \(B\). Since \(A\) is positive semidefinite and nowhere vanishing, it follows that \(\rank \Phi_A = 1\) on \(B\). Hence \zcref{appendix:lem:rank_one_case} yields a non-zero \(1\)-form \(\omega \in \Omega_c^1(B)\) satisfying \(L_A\omega = 0\).

   Extending \(\omega\) by zero from \(B\) to \(D\) proves the claim. Since \(B \subset D\) was chosen arbitrarily small, the support may be chosen in an arbitrarily small disc.
\end{proof}

\begin{lemma}\label{appendix:lem:rank_0-1_dense}
    Let \(s : \Omega \to \R^d\), \(d \geq 4\), be a smooth strictly saddle immersion, and let \(Y \in \Gamma(s^*(T\R^d))\) be an infinitesimal isometric deformation of \(s\). Then the bundle map
    \begin{equation*}
        DY : T\Omega \to s^*(T\R^d)
    \end{equation*}
    cannot have constant rank \(1\) on a non-empty open set.
\end{lemma}
\begin{proof}
    Suppose, to the contrary, that \(DY\) has constant rank \(1\) on a non-empty open set. Using the canonical trivialisation of \(T\R^d\), the Rank Theorem \cite[Theorem 4.12]{lee2012introduction} yields, after restricting to a smaller open set, local coordinates \((x,y)\) and a regular curve
    \begin{equation*}
        \gamma : I \to \R^d \quad \text{such that} \quad Y(x,y) = \gamma(x).
    \end{equation*}
    Hence, we find
    \begin{equation*}
        Y_y = 0 \quad \text{and} \quad \nu := Y_x = \gamma'(x) \neq 0.
    \end{equation*}
    Since \(\mathcal{E}(Y) = 0\), we have
    \begin{equation*}
        0 = \mathcal{E}(Y)(\partial_x, \partial_x) = \langle s_x, \nu\rangle \quad \text{and} \quad 0 = \mathcal{E}(Y)(\partial_x,\partial_y) = \frac{1}{2}\langle s_y, \nu\rangle
    \end{equation*}
    Thus \(\nu\) is a nowhere vanishing normal section. Since \(\nu = \gamma'(x)\), we have \(\nu_y = 0\). Hence, the Weingarten equation gives
    \begin{equation*}
        Ds(W_\nu\partial_y) = -(\nu_y)^\top = 0.
    \end{equation*}
    Since \(s\) is an immersion, this implies \(W_\nu \partial_y = 0\), therefore
    \begin{equation*}
        II^\nu(\partial_y, Z) = g(W_\nu \partial_y, Z) = 0 \quad \text{for every} \ Z \in \Gamma(T\Omega).
    \end{equation*}
    Thus \(II^\nu\) is degenerate, contradicting the strict saddle condition.
\end{proof}

\begin{proof}[Proof of \zcref{lem:rank_0-1_dense}]
    The set of points at which \(DY\) has constant rank in a neighbourhood is open and dense in \(\Omega\). Indeed, every non-empty open subset contains a point at which the rank is maximal on that subset. A corresponding non-zero minor remains non-zero in a smaller neighbourhood, on which the rank is therefore constant. By \zcref{appendix:lem:rank_0-1_dense}, this rank cannot be \(1\). Since \(\dim T_p\Omega = 2\)  for every \(p \in \Omega\), the only remaining possibilities are \(0\) and \(2\), which proves the first assertion.

    Now suppose that \(\Omega\) is connected and \(\{p \in \Omega : \rank (DY)_p = 2\} = \emptyset\). By the first assertion, \(DY = 0\) on a dense subset of \(\Omega\), so, by continuity, \(DY = 0\) on \(\Omega\). Thus \(Y : \Omega \to \R^d\) is constant. Since \(Y\) is compactly supported, it follows that \(Y = 0\), a contradiction.
\end{proof}

\begin{proof}[Proof of \zcref{lem:existence_rank_2_infinitesimal_isometric_deformation}]
    Choose a disc \(\D \subset U\). Since \(\rank \mathcal{B} = 2\) on \(\D\), \(\mathcal{A}\restrict{\D} \subset \Sym^2(T\Omega)\restrict{\D}\) is a line bundle. After shrinking \(\D\), we choose a nowhere vanishing section \(A \in \Gamma(\mathcal{A}\restrict{\D})\). As in the proof of \zcref{prop:existence_infinitesimal_deformation}, the saddle property implies that \(A\) is nowhere indefinite.

    Suppose first that \(A\) is definite at some point. After shrinking \(\D\) around this point, we may assume that \(A\) is definite on \(\D\). For every \(p \in \D\), we have
    \begin{equation*}
        \Ima(\mathcal{B})_p = \{H \in \Sym^2(T_p^*\Omega) : \langle A_p, H\rangle_c = 0\}.
    \end{equation*}
    Since \(A_p\) is definite, every non-zero element of this space is indefinite. Moreover, \(\mathcal{B}\) is injective, since both \((N_s \Omega)_p\) and \(\Ima(\mathcal{B})_p\) have dimension \(2\). Hence \(II^\nu_p\) is indefinite for every non-zero normal vector \(\nu\), so \(s\) is strictly saddle on \(\D\). By \zcref{prop:existence_infinitesimal_deformation}, there exists a non-trivial compactly supported infinitesimal isometric deformation \(Y\) supported in \(\D\). Since \(\D\) is connected, \zcref{lem:rank_0-1_dense} yields \(\rank (DY)_p = 2\) for some \(p \in \D\).

    We may therefore assume that \(A\) is nowhere definite. It follows that \(A\) has rank one on \(\D\) since \(A\) is nowhere vanishing and nowhere indefinite. By replacing \(A\) by \(-A\), if necessary, we may assume that \(A\) is positive semidefinite. By \zcref{appendix:lem:rank_one_case}, there exists \(p \in D\) and a non-zero \(\omega \in \Omega_c^1(D)\) such that, for \(S := \sym(\nabla \omega)\), we have
    \begin{equation*}
        L_A \omega = 0, \quad \omega_p = 0, \quad S_p = 0 \quad \text{and} \quad \rank(\nabla S)_p = 2.
    \end{equation*}
    Since \(A\) generates \(\mathcal{A}\), the equation \(L_A\omega = 0\) yields \(S \in \Ima \mathcal{B}\restrict{\D}\). Because \(\rank \mathcal B\restrict{\D} = 2\) and the normal bundle of \(s : \Omega \to \R^4\) has rank \(2\), 
    \begin{equation*}
        \mathcal{B}:N_s \Omega \restrict{\D} \to \Ima \mathcal{B}\restrict{\D}
    \end{equation*}
    is a vector bundle isomorphism. Hence, there exists a unique smooth normal vector field \(\nu\) satisfying \(\mathcal{B}(\nu) = S\). Define
    \begin{equation*}
        V := \omega^\sharp \quad \text{and} \quad Y := Ds(V) + \nu.
    \end{equation*}
    Then \(\mathcal{E}(Y) = 0\), that is, \(Y\) is an infinitesimal isometric deformation. It remains to prove the rank assertion. At the point \(p \in D\) chosen above, since \(S_p = 0\) and \(\mathcal{B}_p\) is injective, we find \(\nu_p = 0\). By differentiating \(\mathcal{B}(\nu)\) at \(p\), we obtain
    \begin{equation*}
        (\nabla_X S)_p = (\nabla_X \mathcal{B})_p(\nu_p)  + \mathcal{B}_p(\nabla_X^\perp \nu)_p = \mathcal{B}_p(\nabla_X^\perp \nu)_p \quad \text{for} \ X \in T_p\Omega,
    \end{equation*}
    since the term involving \(\nu_p\) vanishes. As \(\mathcal{B}_p\) is an isomorphism onto its image and \(\rank(\nabla S)_p = 2\), it follows that
    \begin{equation*}
        \rank (\nabla^\perp \nu)_p = 2.
    \end{equation*}
    Finally, \(\omega_p = 0\) implies \(V_p = 0\). Hence, the normal component of \(DY_p(X)\) is precisely
    \begin{equation*}
        (DY_p(X))^\perp = II_p(X,V_p) + (\nabla_X^\perp\nu)_p = (\nabla_X^\perp\nu)_p,
    \end{equation*}
    where we have used the Gauss formula \cite[Theorem 8.2]{lee2018introduction}. Thus the normal projection of \(DY_p\) has rank \(2\). Since \(T_p D\) is two-dimensional, we conclude that
    \begin{equation*}
        \rank(DY)_p = 2.
    \end{equation*}
\end{proof}

\section[Auxiliary Proofs for the Density Theorem]{Auxiliary Proofs for \zcref{sec:the_density_theorem}}\label{appendix:sec:auxiliary_proof_for_the_density_theorem}

\begin{proof}[Proof of \zcref{lem:strictly_saddle_is_an_open_condition}]
    The smooth embedding part follows from \cite[Chapter 2, Theorem 1.4]{hirsch1976differential}. For the strictly saddle part, we compute determinants of scalar second fundamental form with respect to the coordinate basis \((\partial_x,\partial_y)\) of \(T\Omega\). The unit normal bundle of \(s_0\) over \(\overline{\Omega}\) is compact, and the function
    \begin{equation*}
        (p,\nu) \mapsto \det II_{s_0}^\nu(p)
    \end{equation*}
    is continuous and strictly negative on this bundle. Hence there exists \(c > 0\) such that
    \begin{equation*}
        \det II_{s_0}^\nu(p) \leq -c \quad \text{for every} \ p \in \overline{\Omega}, \nu \in (S(N_{s_0}\Omega))_p.
    \end{equation*}
    We claim that, if \(s\) is sufficiently close in \(C^1\)-topology, their normal vector fields are close as well. We consider the orthogonal projection \(P_s^N(p) : \R^4 \to (N_s\Omega)_p\) and
    \begin{equation*}
        \Phi:S(N_{s_0}\Omega) \to S(N_{s}\Omega), \ \Phi_s(p, \nu) = \frac{P_s^N(p)\nu}{|P_s^N(p)\nu|}
    \end{equation*}
    is an isomorphism for every \(p \in \overline{\Omega}\). Under this identification, the function \((p, \nu) \mapsto \det II_s^{\Phi_s(p,\nu)}(p)\) depends continuously on \(s\) in the \(C^2\)-topology. Therefore, after decreasing \(\varepsilon_*\) if necessary, \(\|s-s_0\|_{C^2(\overline{\Omega})} < \varepsilon_*\) implies
    \begin{equation*}
        \det II_s^{\Phi_s(p,\nu)}(p) < -\frac{c}{2} < 0 \quad \text{for every} \ (p,\nu) \in S(N_{s_0}\Omega).
    \end{equation*}
    Since \(\Phi_s\) parametrises all unit normal vectors of \(s\), every non-zero scalar second fundamental form of \(s\) is indefinite. So \(s\) is strictly saddle.
    
\end{proof}

\printbibliography

\bigskip

\begin{small}
	\noindent
	Stefan Christian Kohlmeier\\
	Institute of Mathematics, Paderborn University\\
	\textit{E-mail:} {\hypersetup{allcolors=black}\href{mailto:stekoh@mail.uni-paderborn.de}{\texttt{stekoh@mail.uni-paderborn.de}}}
\end{small}

\end{document}